\documentclass[12pt,twoside,reqno]{amsart}
\usepackage{amsmath}
\usepackage{amsfonts}
\usepackage{amssymb}
\usepackage{amsthm}
\usepackage{color}
\usepackage{mathrsfs}
\usepackage{cite}
\newtheorem{theorem}{Theorem}[section]
\newtheorem{corollary}[theorem]{Corollary}
\newtheorem{lemma}[theorem]{Lemma}
\newtheorem{proposition}[theorem]{Proposition}
\newtheorem{definition}[theorem]{Definition}
\newtheorem{remark}[theorem]{Remark}
\numberwithin{equation}{section}
\begin{document}

\title{Weak compactness techniques and coagulation equations}
\author{Philippe Lauren{\c c}ot}
\address{Institut de Math{\'e}matiques de Toulouse, UMR~5219, Universit{\'e} de Toulouse, CNRS, F--31062 Toulouse Cedex 9, France}
\email{laurenco@math.univ-toulouse.fr}

\begin{abstract}
Smoluchowski's coagulation equation is a mean-field model describing the growth of clusters by successive mergers. Since its derivation in 1916 it has been studied by several authors, using deterministic and stochastic approaches, with a blossoming of results in the last twenty years. In particular, the use of weak $L^1$-compactness techniques led to a mature theory of weak solutions and the purpose of these notes is to describe the results obtained so far in that direction, as well as the mathematical tools used.
\end{abstract}
%
%
\maketitle
%
%
\pagestyle{myheadings}
\markboth{\sc{Ph. Lauren\c cot}}{\sc{Weak compactness and coagulation}}

\section{Introduction}
\label{sec:1}

Coagulation is one of the driving mechanisms for cluster growth, by which clusters (or particles) increase their sizes by successive mergers. Polymer and colloidal chemistry, aerosol science, raindrops and soot formation, astrophysics (formation of planets and galaxies), hematology, and animal herding are among the fields where coagulation phenomena play an important role, see  \cite{Dr72, Fr00, GL95, SP06} for instance. This variety of applications has generated a long lasting interest in the modeling of coagulation processes. One of the first contributions in that direction is due to the Polish physicist Smoluchowski who derived a model for the evolution of a population of colloidal particles increasing their sizes by binary coagulation while moving according to independent Brownian motion \cite{Sm16,Sm17}. Neglecting spatial variations he came up with the discrete Smoluchowski coagulation equations 
\begin{align}
\frac{df_1}{dt} & = - \sum_{j=1}^\infty K(1,j) f_1 f_j\ , \quad t>0\ , \label{in1} \\
\frac{df_i}{dt} & = \frac{1}{2} \sum_{j=1}^{i-1} K(j,i-j) f_{i-j} f_j - \sum_{j=1}^\infty K(i,j) f_i f_j\ , \quad i\ge 2\ , \ t>0 \ . \label{in2}
\end{align}
Here the sizes of the particles are assumed to be multiples of a minimal size normalized to one and the coagulation kernel $K(i,j)$ accounts for the rate at which a particle of size $i$ and a particle of size $j$ encounter and merge into a single particle of size $i+j$. In the derivation performed in \cite{Sm16, Sm17} $K$ is computed to be
\begin{equation}
K_{sm}(x,y) := \left( x^{1/3} + y^{1/3} \right) \left( \frac{1}{x^{1/3}} + \frac{1}{y^{1/3}} \right) \label{in3}
\end{equation}
and then reduced to $K_0(x,y):=2$ to allow for explicit computations of solutions to \eqref{in1}-\eqref{in2}. 

The function $f_i$, $i\ge 1$, denotes the size distribution function of particles of size $i\ge 1$ at time $t\ge 0$ and the meaning of the reaction terms in \eqref{in1}-\eqref{in2} is the following: the second term on the right-hand side of \eqref{in2} describes the depletion of particles of size $i$ by coalescence with other particles of arbitrary size while the first term on the right-hand side of \eqref{in2} accounts for the gain of particles of size $i$ due to the merging of a particle of size $j\in\{1,\cdots,i-1\}$ and a particle of size $i-j$. Note that the assumption of a minimal size entails that there is no formation of particles of size $1$ by coagulation. Schematically, if $P_x$ stands for a generic particle of size $x$, the coagulation events taken into account in the previous model are:
\begin{equation}
P_x + P_y \longrightarrow P_{x+y}\ . \label{in4}
\end{equation}
The formation of particles of size $i$ corresponds to the choice $(x,y)=(j,i-j)$ in \eqref{in4} with $i\ge 2$ and $1 \le j \le i-1$ and the disappearance of particles of size $i$ to $(x,y)=(i,j)$ in \eqref{in4} with $i\ge 1$ and $j\ge 1$. A salient feature of the elementary coagulation reaction \eqref{in4} is that no matter is lost and we shall come back to this point later on. 

Smoluchowski's coagulation equation was later on extended to a continuous size variable $x\in (0,\infty)$ and reads \cite{Mu28}
\begin{eqnarray}
\partial_t f(t,x) & = & \frac{1}{2} \int_0^x K(y,x-y)\ f(t,y)\ f(t,x-y)\ dy \nonumber\\
& & \ - \int_0^\infty K(x,y)\ f(t,x)\ f(t,y)\ dy\,, \quad (t,x)\in (0,\infty)\times (0,\infty)\ . \label{in5} 
\end{eqnarray}
In contrast to \eqref{in1}-\eqref{in2} which is a system of countably many ordinary differential equations, equation \eqref{in5} is a nonlinear and nonlocal integral equation but the two terms of the right-hand side of \eqref{in5} have the same physical meaning as in \eqref{in1}-\eqref{in2}. The coagulation kernel $K(x,y)$ still describes the likelihood that a particle of mass $x>0$ and a particle of mass $y>0$ merge into a single particle of mass $x+y$ according to \eqref{in4}. Besides Smoluchowski's  coagulation kernel \eqref{in3} \cite{Sm16, Sm17} and the constant coagulation kernel $K_0(x,y)=2$, other coagulation kernels have been derived in the literature such as $K(x,y) = (ax+b) (ay+b)$, $a>0$, $b\ge 0$ \cite{So43}, $K(x,y) = \left( x^{1/3} + y^{1/3}\right)^{3}$, $K(x,y) = \left( x^{1/3} + y^{1/3}\right)^{2}\ \left| x^{1/3} - y^{1/3}\right|$, and $K(x,y) = x^\alpha y^\beta + x^\beta y^\alpha$, $\alpha\le 1$, $\beta\le 1$, the latter being rather a model case which includes the constant coagulation kernel $K_0$ ($\alpha=\beta=0$), the additive one $K_1(x,y):=x+y$ ($\alpha=0, \beta=1$), and the multiplicative one $K_2(x,y):=xy$ ($\alpha=\beta=1$). Observe that this short list of coagulation kernels already reveals a wide variety of behaviours for large values or small values of $(x,y)$ (bounded or unbounded) and on the diagonal $x=y$ (positive or vanishing). 

A central issue is which predictions on the coagulation dynamics can be made from the analysis of Smoluchowski's coagulation equation \eqref{in5} and to what extent these predictions depend upon the properties of the coagulation kernel $K$. It was uncovered several years ago that the evolution of \eqref{in5} could lead to two different dynamics according to the properties of $K$ and is closely related to the conservation of matter already alluded to. More precisely, recall that there is neither loss nor gain of matter during the elementary coagulation reaction \eqref{in4} and this is expected to be true as well during the time evolution of \eqref{in5}. In terms of $f$, the total mass of the particles distribution at time $t\ge 0$ is 
$$
M_1(f(t)) := \int_0^\infty x f(t,x)\ dx
$$
and mass conservation reads 
\begin{equation}
M_1(f(t)) = M_1(f(0))\ , \quad\quad t\ge 0\ , \label{in6}
\end{equation}
provided $M_1(f(0))$is finite. To check whether \eqref{in6} is true or not, we first argue formally and observe that, if $\vartheta$ is an arbitrary function, multiplying \eqref{in5} by $\vartheta(x)$ and integrating with respect to $x\in (0,\infty)$ give, after exchanging the order of integration,
\begin{align}
& \frac{d}{dt} \int_0^\infty \vartheta(x) f(t,x)\ dx \nonumber \\ 
& \qquad = \frac{1}{2} \int_0^\infty \int_0^\infty \left[ \vartheta(x+y) - \vartheta(x) - \vartheta(y) \right] K(x,y) f(t,x) f(t,y)\ dydx \label{in7}
\end{align}
Clearly, the choice $\vartheta(x)=x$ leads to a vanishing right-hand side of \eqref{in7} and provides the expected conservation of mass. However, one has to keep in mind that the previous computation is only formal as it uses Fubini's theorem without justification. That some care is indeed needed stems from \cite{LT81} where it is shown that \eqref{in6} breaks down in finite time for the multiplicative kernel $K_2(x,y)=xy$ for all non-trivial solutions. An immediate consequence of this result is that the mass-conserving solution constructed on a finite time interval in \cite{ML62a, ML62b} cannot be extended forever. Soon after the publication of \cite{LT81} a particular solution to \eqref{in1}-\eqref{in2} was constructed for $K(i,j)=(ij)^\alpha$, $\alpha\in (1/2,1)$ which fails to satisfy \eqref{in6} for all times \cite{Le83}. At the same time, it was established in \cite{LT82, Wh80} that the condition $K(i,j)\le \kappa (i+j)$ was sufficient for \eqref{in1}-\eqref{in2} to have global mass-conserving solutions, that is, solutions satisfying \eqref{in6}. Thanks to these results, a distinction was made between the so-called non-gelling kernels for which all solutions to \eqref{in1}-\eqref{in2} and \eqref{in5} satisfy \eqref{in6} and gelling kernels for which \eqref{in6} is infringed in finite time for all non-trivial solutions. The conjecture stated in the beginning of the eighties is that a coagulation kernel satisfying 
\begin{equation}
K(x,y) \le \kappa (2 + x + y)\ , \quad (x,y)\in (0,\infty)\times (0,\infty)\ , \label{in8} 
\end{equation}
is non-gelling while a coagulation kernel satisfying
\begin{equation}
K(x,y) \ge \kappa (xy)^\alpha\ , \quad \alpha>\frac{1}{2}\ , \quad (x,y)\in (0,\infty)\times (0,\infty)\ , \label{in9} 
\end{equation}
is gelling \cite{EZH84, HEZ83, LT82}. Following the contributions \cite{LT82, Wh80} the conjecture for non-gelling coagulation kernels \eqref{in8} is completely solved in \cite{BC90} for the discrete coagulation equations \eqref{in1}-\eqref{in2}, an alternative proof being given in \cite{Lt02}. It took longer for this conjecture to be solved for the continuous coagulation equation \eqref{in5}, starting from the pioneering works \cite{AB79, Me57} for bounded kernels and continuing with \cite{DS96, LM02, St89, St91}. Of particular importance is the contribution by Stewart \cite{St89} where weak $L^1$-compactness techniques were used for the first time and turned out to be a very efficient tool which was extensively used in subsequent works. As for the conjecture for gelling kernels \eqref{in9}, it was solved rather recently in \cite{EMP02, Je98}. An intermediate step is the existence of solutions to \eqref{in1}-\eqref{in2} and \eqref{in5} with non-increasing finite mass, that is, satisfying $M_1(f(t))\le M_1(f(0))$ for $t\ge 0$, see \cite{EW01, Lt00, LM02, LT81, No99, Sp84} and the references therein. 

\medskip

The purpose of these notes is twofold: on the one hand, we collect in Section~\ref{sec:2} several results on the weak compactness in $L^1$-spaces which are scattered throughout the literature and which have proved useful in the analysis of \eqref{in5}. We recall in particular the celebrated Dunford-Pettis theorem (Section~\ref{subsec:22}) which characterizes weakly compact sequences in $L^1$ with the help of the notion of uniform integrability (Section~\ref{subsec:23}). Several equivalent forms of the latter are given, including a refined version of the de la Vall\'ee Poussin theorem \cite{DM75, Le77, VP15} (Section~\ref{subsec:24}). We also point out consequences of the combination of almost everywhere convergence and weak convergence (Section~\ref{subsec:25}). On the other hand, we show in Section~\ref{sec:3} how the results stated in Section~\ref{sec:2} apply to Smoluchowski's coagulation equation \eqref{in5} and provide several existence results including that of mass-conserving solutions (Section~\ref{sec:32}). For the sake of completeness, we supplement the existence results with the occurrence of gelation in finite time for gelling kernels \cite{EMP02} (Section~\ref{sec:33}) and with uniqueness results (Section~\ref{sec:34}). For further information on coagulation equations and related problems we refer to the books \cite{Bt06, Du94b} and the survey articles \cite{Al99, LM04, Le03, Wa06}.

\medskip

We conclude the introduction with a few words on related interesting issues: we focus in these notes on the deterministic approach to the modeling of coagulation and leave aside the stochastic approach which has been initiated in \cite{Lu78, Ma68, Sm16, Sm17} and further developed in \cite{Al99, Bt02, Bt06, DFT02, DGG99, EW01, FG04, Je98, Jo03, No99} and the references therein. 

Another important line of research is the dynamics predicted by Smoluchowski's coagulation equation \eqref{in5} for large times for homogeneous non-gelling kernels \eqref{in8} and at the gelation time for homogeneous gelling kernels \eqref{in9}. In both cases the expected behaviour is of self-similar form (except for some particular kernels with homogeneity $1$) but the time and mass scales are only well identified for non-gelling kernels and for the multiplicative kernel $K_2(x,y)=xy$, see the survey articles \cite{vDE88, Le03} and the references therein. Existence of mass-conserving self-similar solutions for a large class of non-gelling kernels have been constructed recently \cite{EMRR05, FL05} and their properties studied in \cite{CM11, EM06, FL06a, MLNV11, NV11}. Still for non-gelling kernels, the existence of other self-similar solutions (with a different scaling and possibly infinite mass) is uncovered in \cite{Bt02} for the additive kernel $K_1(x,y) = x+y$ and in \cite{MP04} for the constant kernel $K_0(x,y)=2$, both results relying on the use of the Laplace transform which maps \eqref{in5} either to Burgers' equation or to an ordinary differential equation. Since the use of the Laplace transform has not proved useful for other coagulation kernels, a much more involved argument is needed to cope with a more general class of kernels \cite{NV13a, NV13b}. 

Finally, coagulation is often associated with the reverse process of fragmentation and there are several results available for coagulation-fragmentation equations, including existence, uniqueness, mass conservation, and gelation. Actually, the approach described below in Section~\ref{sec:3} works equally well for coagulation-fragmentation equations under suitable assumptions on the fragmentation rates. Besides the survey articles \cite{Du94b, LM04, Wa06}, we refer for instance to \cite{BC90, dC95, Je98, Jo03, Lt02, Sp84} for the discrete coagulation-fragmentation equations and to \cite{AB79, BL11, BL12, DS96, EW01, EMP02, ELMP03, Gi13, GW11, GLW12, La04, Lt00, LM02, MLLM97, Me57, St89, St90} for the continuous coagulation-fragmentation equations.

\section{Weak compactness in $L^1$}
\label{sec:2}

Let $(\Omega,\mathcal{B},\mu)$ be a $\sigma$-finite measure space. For $p\in [1,\infty]$, $L^p(\Omega)$ is the usual Lebesgue space and we denote its norm by $\|\cdot\|_p$. If $p\in (1,\infty)$, the reflexivity of the space $L^p(\Omega)$ warrants that any bounded sequence in $L^p(\Omega)$ has a weakly convergent subsequence. In the same vein, any bounded sequence in $L^\infty(\Omega)$ has a weakly-$\star$ convergent subsequence by a consequence of the Banach-Alaoglu theorem \cite[Corollary~3.30]{Br11} since $L^\infty(\Omega)$ is the dual of the separable space $L^1(\Omega)$ . A peculiarity of $L^1(\Omega)$ is that a similar property is not true as a consequence of the following result \cite[Appendix to Chapter~V, Section~4]{Yo95}, the space $L^1(\Omega)$ being not reflexive. 

\begin{theorem}[Eberlein-\v Smulian]\label{theb}
Let $E$ be a Banach space such that every bounded sequence has a subsequence converging in the $\sigma(E,E')$-topology. Then $E$ is reflexive.
\end{theorem}

\subsection{Failure of weak compactness in $L^1$}
\label{subsec:21}

In a simpler way, a bounded sequence in $L^1(\Omega)$ need not be weakly sequentially compact  in $L^1(\Omega)$ as the following examples show:

\begin{description}
\item[\textbf{Concentration}] Consider $f\in C^\infty(\mathbb{R}^N)$ such that $f\ge 0$, $\mbox{ supp }f \subset B(0,1):= \{ x\in\mathbb{R}^N\ :\ |x|<1\}$, and $\|f\|_1=1$. For $n\ge 1$ and $x\in\mathbb{R}^N$, we define $f_n(x) = n^N f(nx)$ and note that 
\begin{equation}
\label{a1}
\|f_n\|_1=\|f\|_1=1\,.
\end{equation}
Thus $(f_n)_{n\ge 1}$ is bounded in $L^1(\mathbb{R}^N)$. Next, the function 
$f$ being compactly supported, we have
$$
\lim_{n\to \infty} f_n(x) = 0 \;\;\mbox{ for }\;\ x\ne 0\,.
$$
The only possible weak limit of $(f_n)_{n\ge 1}$ in $L^1(\mathbb{R}^N)$ would then be zero which contradicts \eqref{a1}. Consequently, $(f_n)_{n\ge 1}$ has no cluster point in the weak topology of $L^1(\mathbb{R}^N)$. In fact, $(f_n)_{n\ge 1}$ converges narrowly towards a bounded measure, the Dirac mass, that is,
$$
\lim_{n\to \infty} \int_{\mathbb{R}^N} f_n(x)\ \psi(x)\ dx = \psi(0) \;\;\mbox{ for all }\;\; \psi\in BC(\mathbb{R}^N)\,,
$$
where $BC(\mathbb{R}^N)$ denotes the space of bounded and continuous functions on $\mathbb{R}^N$. The sequence $(f_n)_{n\ge 1}$ is not weakly sequentially compact in $L^1(\mathbb{R}^N)$ because it \textsl{concentrates} in the neighbourhood of $x=0$. Indeed, for $r>0$, we have
$$
\int_{\{|x|\le r\}} f_n(x)\ dx = \int_{\{|x|\le n
r\}} f(x)\ dx \mathop{\longrightarrow}_{n \to \infty} 1\,,
$$ 
and
$$
\int_{\{|x|\ge r\}} f_n(x)\ dx = \int_{\{|x|\ge n
r\}} f(x)\ dx \mathop{\longrightarrow}_{n \to \infty} 0\,.
$$ 

\item[\textbf{Vanishing}] For $n\ge 1$ and $x\in\mathbb{R}$, we set
$f_n(x) = \exp{\left( -|x-n| \right)}$. Then 
\begin{equation}
\label{a2}
\|f_n\|_1= 2\,,
\end{equation}
and $(f_n)_{n\ge 1}$ is bounded in $L^1(\mathbb{R})$. Next, 
$$
\lim_{n\to \infty} f_n(x) = 0 \;\;\mbox{ for all }\;\; x\in\mathbb{R}\,,$$ 
and we argue as in the previous example to conclude that $(f_n)_{n\ge 1}$ has no cluster point for the weak topology of $L^1(\mathbb{R})$. In that case, the sequence $(f_n)_{n\ge 1}$ ``escapes at infinity'' in the sense that, for every $r>0$,
$$
\lim_{n\to \infty} \int_r^\infty f_n(x)\ dx = 2 \;\;\mbox{ and
}\;\; \lim_{n\to \infty} \int_{-\infty}^r f_n(x)\ dx = 0\,.
$$ 
\end{description}

The above two examples show that the mere boundedness of a sequence in $L^1(\Omega)$ does not guarantee at all its weak sequential compactness and additional information is thus required for the latter to be true. It actually turns out that, roughly speaking, the only phenomena that prevent a bounded sequence in $L^1(\Omega)$ from being weakly sequentially compact in $L^1(\Omega)$ are the concentration and vanishing phenomena described in the above examples. More precisely, a necessary and sufficient condition for the weak sequential compactness in $L^1(\Omega)$ of a bounded sequence in $L^1(\Omega)$ is given by the Dunford-Pettis theorem which is recalled below. 
 
\subsection{The Dunford-Pettis theorem}
\label{subsec:22}

We first introduce the modulus of uniform integrability of a bounded subset $\mathcal{F}$ of $L^1(\Omega)$ which somehow measures how elements of $\mathcal{F}$ concentrate on sets of small measures. 

\begin{definition}\label{deb100}
Let $\mathcal{F}$ be a bounded subset of $L^1(\Omega)$. 
For $\varepsilon>0$, we set
\begin{equation}
\label{b101}
\eta\{\mathcal{F},\varepsilon\} := \sup{ \left\{ \int_A |f|\ d\mu\ : \  f\in \mathcal{F}\,, \; A\in \mathcal{B}\,, \; \mu(A)\le \varepsilon \right\} }\,,
\end{equation}
and we define the modulus of uniform integrability $\eta\{\mathcal{F}\}$ of $\mathcal{F}$ by
\begin{equation}
\label{b102}
\eta\{\mathcal{F}\} := \lim_{\varepsilon\to 0} \eta\{\mathcal{F},\varepsilon\} = \inf_{\varepsilon>0} \eta\{\mathcal{F},\varepsilon\}\,. 
\end{equation}
\end{definition}

With this definition, we can state the Dunford-Pettis theorem, see \cite[Part~2, Chapter~VI, \S~2]{Be85}, \cite[p.~33--44]{DM75}, and \cite[IV.8]{DS57} for instance.  

\begin{theorem}\label{thb103}
Let $\mathcal{F}$ be a subset of $L^1(\Omega)$. 
The following two statements are equivalent:
\begin{itemize}
\item[\textbf{(a)}] $\mathcal{F}$ is relatively weakly sequentially compact in $L^1(\Omega)$.
\item[\textbf{(b)}] $\mathcal{F}$ is a bounded subset of $L^1(\Omega)$ 
satisfying the following two properties: 
\begin{equation}
\label{b104}
\eta\{\mathcal{F}\}=0\,,
\end{equation}
and, for every $\varepsilon>0$, there is $\Omega_\varepsilon\in\mathcal{B}$ such that $\mu(\Omega_\varepsilon)<\infty$ and
\begin{equation}
\label{b105}
\sup_{f\in \mathcal{F}}\ \int_{\Omega\setminus \Omega_\varepsilon} 
|f|\ d\mu \le \varepsilon.
\end{equation}
\end{itemize}
\end{theorem}

As already mentioned, the two conditions required in Theorem~\ref{thb103}~(b) to guarantee the weak sequential compactness of $\mathcal{F}$ in $L^1(\Omega)$ exclude the concentration and vanishing phenomena: indeed, the condition \eqref{b104} implies that no concentration can take place while \eqref{b105} prevents the escape to infinity which arises in the vanishing phenomenon.

\begin{remark} \label{reb106} The condition \eqref{b105} is automatically fulfilled as soon as $\mu(\Omega)<\infty$ (with $\Omega_\varepsilon=\Omega$ for each $\varepsilon>0$).
\end{remark}

Thanks to the Dunford-Pettis theorem, the weak sequential compactness of a subset of $L^1(\Omega)$ can be checked by investigating the behaviour of its elements on measurable subsets of $\Omega$. However, the characteristic functions of these sets being not differentiable, applying the Dunford-Pettis theorem in the field of partial differential equations might be not so easy. Indeed, as (partial) derivatives are involved, test functions are usually required to be at least weakly differentiable (or in a Sobolev  space) which obviously excludes characteristic functions. Fortunately, an alternative formulation of the condition $\eta\{\mathcal{F}\}=0$ is available and turns out to be more convenient to use in this field.

\subsection{Uniform integrability in $L^1$}
\label{subsec:23}

\begin{definition} \label{deb200}
A subset $\mathcal{F}$ of $L^1(\Omega)$ is said to be uniformly integrable if $\mathcal{F}$ is a bounded subset of $L^1(\Omega)$ such that 
\begin{equation}
\label{b201}
\lim_{c\to \infty}\ \sup_{f\in \mathcal{F}}\ \int_{\{ |f|\ge c \}} 
|f|\ d\mu = 0\ .
\end{equation}
\end{definition}

Before relating the uniform integrability property with the weak sequential compactness in $L^1(\Omega)$, let us give some simple examples of uniformly integrable subsets:
\begin{itemize}
\item If $\mathcal{F}$ is a bounded subset of $L^p(\Omega)$ for some $p\in(1,\infty)$, then $\mathcal{F}$ is uniformly integrable as 
$$
\sup_{f\in \mathcal{F}}\ \int_{\{ | f|\ge c \}} | f|\ d\mu \le \frac{1}{c^{p-1}}\ \sup_{f\in \mathcal{F}}\ \int_{\{ | f|\ge c \}} | f|^p\ d\mu \le \frac{1}{c^{p-1}}\ \sup_{f\in \mathcal{F}}\{ \| f\|_p^p\}\ .
$$
\item If $f_0\in L^1(\Omega)$, the set $\mathcal{F} := \left\{ f\in L^1(\Omega)\ : \ | f| \le | f_0| \;\; \mu-\mbox{a.e. } \right\}$ is uniformly integrable.
\item If $\mathcal{F}$ is a uniformly integrable subset of $L^1(\Omega)$, then so is the set $\mathcal{F}^*$ defined by  $\mathcal{F}^* := \left\{ | f|\ : \ f\in\mathcal{F} \right\}$.
\item If $\mathcal{F}$ and $\mathcal{G}$ are uniformly integrable subsets of $L^1(\Omega)$, then so is the set $\mathcal{F}+\mathcal{G}$ defined by  $\mathcal{F}+\mathcal{G} := \left\{ f+g\ : \ (f,g)\in\mathcal{F}\times \mathcal{G} \right\}$.
\end{itemize}

We next state the connection between the Dunford-Pettis theorem and the uniform integrability property.

\begin{proposition} \label{prb202}
Let $\mathcal{F}$ be a subset of $L^1(\Omega)$. 
The following two statements are equivalent:
\begin{itemize}
\item[\textbf{(i)}] $\mathcal{F}$ is uniformly integrable.
\item[\textbf{(ii)}] $\mathcal{F}$ is a bounded subset of $L^1(\Omega)$ 
such that $\eta\{\mathcal{F}\}=0$.
\end{itemize}
\end{proposition}

In other words, the uniform integrability property prevents the concentration on sets of arbitrary small measure. Proposition~\ref{prb202} is a straightforward consequence of the following result \cite{Sl85}.

\begin{lemma}\label{leb203}
Let $\mathcal{F}$ be a bounded subset of $L^1(\Omega)$. Then
\begin{equation}
\label{b204}
\eta\{\mathcal{F}\} = \lim_{c\to \infty}\ \sup_{f\in \mathcal{F}}\ \int_{\{ | f|\ge c \}} 
| f|\ d\mu\,.
\end{equation}
\end{lemma}

\begin{proof} We put 
$$
\eta_\star := \lim_{c\to \infty}\ \sup_{f\in \mathcal{F}}\ \int_{\{ | f|\ge c \}} | f|\ d\mu = \inf_{c\ge 0}\ \sup_{f\in \mathcal{F}}\ \int_{\{ | f|\ge c \}} 
| f|\ d\mu\,.
$$ 

We first establish that $\eta\{\mathcal{F}\}\le \eta_\star$. To this end, consider $\varepsilon>0$, $A\in \mathcal{B}$, $g\in \mathcal{F}$, and $c\in (0,\infty)$. If $\mu(A)\le\varepsilon$, we have
\begin{eqnarray*}
\int_A | g|\ d\mu & = & \int_{A\cap \{ | g| < c\}} 
| g|\ d\mu + \int_{\{ A \cap | g| \ge c\}} 
| g|\ d\mu\\
& \le & c\ \mu(A) + \int_{\{ | g| \ge c\}} 
| g|\ d\mu\\
& \le & c\ \varepsilon + \sup_{f\in \mathcal{F}}\ \int_{\{ | f| \ge c\}}  | f|\ d\mu\,,
\end{eqnarray*}
whence
$$
\eta\{\mathcal{F},\varepsilon\} \le c\ \varepsilon + \sup_{f\in \mathcal{F}}\ \int_{\{ | f| \ge c\}} | f|\ d\mu\,.
$$
Passing to the limit as $\varepsilon\to 0$ leads us to 
$$
\eta\{\mathcal{F}\} \le \sup_{f\in \mathcal{F}}\ \int_{\{ | f| \ge c\}} | f|\ d\mu
$$
for all $c\in (0,\infty)$. Letting $c\to \infty$ readily gives the inequality $\eta\{\mathcal{F}\}\le \eta_\star$.

We now prove the converse inequality. For that purpose, we put
$$ 
\Lambda := \sup_{f\in \mathcal{F}} \{\| f\|_1\} <\infty \,,
$$
and observe that
$$
\mu\left({ \left\{{ x\in\Omega\ , \ | f(x)|\ge c }\right\} 
}\right) \le \frac{\Lambda}{c}
$$
for all $f\in \mathcal{F}$ and $c>0$. Consequently,
$$
\sup_{f\in \mathcal{F}}\ \mu\left({ \left\{{ x\in\Omega\ , \ | f(x)|\ge c }\right\} }\right) \le \frac{\Lambda}{c}\,,
$$
from which we deduce that
$$
\eta_\star \le \sup_{f\in \mathcal{F}}\ \int_{\{ | f|\ge c \}} 
| f|\ d\mu \le \eta\left\{ \mathcal{F},\frac{\Lambda}{c} \right\}\,.
$$
Since the right-hand side of the previous inequality converges towards $\eta\{\mathcal{F}\}$ as $c\to \infty$, we conclude that $\eta_\star\le 
\eta\{\mathcal{F}\}$ and complete the proof of the lemma. 
\end{proof}

Owing to Proposition~\ref{prb202}, the property $\eta\{\mathcal{F}\}=0$ can now be checked by studying the sets where the elements of $\mathcal{F}$ reach large values. This turns out to be more suitable in the field of partial differential equations as one can use the functions $r\longmapsto (r-c)_+:=\max{\{ 0, r-c \}}$. For instance, if $\Omega$ is an open set of $\mathbb{R}^N$ and $u$ is a function in the Sobolev space $W^{1,p}(\Omega)$ for some $p\in [1,\infty]$, then $(u-c)_+$ has the same regularity with $\nabla (u-c)_+ = \text{sign}((u-c)_+) \nabla u$. This allows one in particular to use $(u-c)_+$ as a test function in the weak formulation of nonlinear second order elliptic and parabolic equations and thereby obtain useful estimates. A broader choice of functions is actually possible as we will see in the next theorem.

\subsection{The de la Vall\'ee Poussin theorem}
\label{subsec:24}

\begin{theorem} \label{thb106}
Let $\mathcal{F}$ be a subset of $L^1(\Omega)$. The following two statements are equivalent:
\begin{itemize}
\item[\textbf{(i)}] $\mathcal{F}$ is uniformly integrable.
\item[\textbf{(ii)}] $\mathcal{F}$ is a bounded subset of $L^1(\Omega)$ and there exists a convex function $\Phi\in C^\infty([0,\infty))$ 
such that $\Phi(0) = \Phi'(0) = 0$, $\Phi'$ is a concave function, 
\begin{eqnarray}
\label{b107}
& & \Phi'(r)>0 \;\;\mbox{ if }\;\; r> 0,\\
\label{b108}
& & \lim_{r\to \infty} \frac{\Phi(r)}{r} = \lim_{r\to \infty} \Phi'(r) 
= \infty,
\end{eqnarray}
and
\begin{equation}
\label{b109}
\sup_{f\in \mathcal{F}} \int_\Omega \Phi\left({ | f| }\right)\ d\mu 
<\infty.
\end{equation}
\end{itemize}
\end{theorem}

When $\mathcal{F}$ is a sequence of integrable functions $(f_n)_{n\ge 1}$, Theorem~\ref{thb106} is established by de la Vall\'ee Poussin \cite[p.~451--452]{VP15} (without the concavity of $\Phi'$ and the regularity of $\Phi$) and is stated as follows: a sequence $(f_n)_{n\ge 1}$ is uniformly integrable (in the sense that the subset $\{ f_n \ : \ n\ge 1\}$ of $L^1(\Omega)$ is uniformly integrable) if and only if there is a non-decreasing function $\varphi~: [0,\infty)\to [0,\infty)$ such that $\varphi(r)\to \infty$ as $r\to \infty$ and 
$$
\sup_{n\ge 1} \int_\Omega \varphi(| f_n|)\ | f_n|\ d\mu < \infty\,.
$$
This result clearly implies Theorem~\ref{thb106}. Indeed, if $\Phi$ denotes the primitive of $\varphi$ satisfying $\Phi(0)=0$, the function $\Phi$ is clearly convex and the convexity inequality $\Phi(r)\le r\ \varphi(r)$ ensures that $(\Phi(| f_n|))_{n\ge 1}$ is bounded in  $L^1(\Omega)$. When $\mu(\Omega)<\infty$, a proof of Theorem~\ref{thb106} may also be found in \cite{DM75} and \cite[Theorem~I.1.2]{RR91} but the first derivative of $\Phi$ is not necessarily concave. As we shall see in the examples below, the possibility of choosing $\Phi'$ concave turns out to be helpful. The version of the de la Vall\'ee Poussin theorem stated in Theorem~\ref{thb106} is actually established in \cite{Le77}. The proof given below is slightly different from those given in the above mentioned references and relies on the following lemma: 

\begin{lemma} \label{leb110}
Let $\Phi\in C^1([0,\infty))$ be a non-negative and convex function with $\Phi(0)=\Phi'(0)=0$ and consider a non-decreasing sequence of integers $(n_k)_{k\ge 0}$ such that $n_0=1$, $n_1\ge 2$, and $n_k\to \infty$ as $k\to\infty$. Given $f\in L^1(\Omega)$ and $k\ge 1$, we have the following inequality:
\begin{eqnarray}
\int_{\{ | f|<n_k\}}  \Phi(| f|)\ d\mu & \le & 
\Phi'(1)\ \int_\Omega | f|\ d\mu \nonumber\\
\label{b111}
& & + \sum_{j=0}^{k-1} \left({ \Phi'(n_{j+1}) - \Phi'(n_j) }\right)\ 
\int_{\{ | f|\ge n_j \}} | f|\ d\mu.
\end{eqnarray}
\end{lemma}

\begin{proof} As $\Phi$ is convex with $\Phi'(0)=0$, $\Phi'$ is non-negative and non-decreasing and 
$$
\Phi(r) \le r\ \Phi'(r), \;\;\; r\in [0,\infty)\,.
$$
Fix $k\ge 1$. We infer from the properties of $\Phi$ that
\begin{eqnarray*}
\int_{\{ |f|<n_k\}} \Phi(|f|)\ d\mu & \le & \int_{\{ |f|<n_k\}} \Phi'(|f|)\ |f|\ d\mu \\
& = & \int_{\{ 0 \le |f| < 1 \}} \Phi'(|f|)\ |f|\ d\mu + \sum_{j=0}^{k-1} \int_{\{ n_j\le |f| < n_{j+1} \}}  \Phi'(|f|)\ |f|\ d\mu\\
& \le & \Phi'(1)\ \int_{\{ 0\le |f| < 1 \}} |f|\ d\mu 
+ \sum_{j=0}^{k-1} \Phi'(n_{j+1})\ \int_{\{ n_j \le |f| < n_{j+1} \}} 
|f|\ d\mu\\
& = & \Phi'(1)\ \int_{\{ 0\le |f| < 1 \}} |f|\ d\mu  
+ \sum_{j=0}^{k-1} \Phi'(n_{j+1})\ \int_{\{ |f|\ge n_j \}} 
|f|\ d\mu \\
& &  - \sum_{j=1}^{k} \Phi'(n_j)\ \int_{\{ |f|\ge n_j \}} |f|\ d\mu\\
& \le & \Phi'(1)\ \int_\Omega |f|\ d\mu \\ 
& & +\ \sum_{j=0}^{k-1} \left({ \Phi'(n_{j+1}) - \Phi'(n_j) }\right)\ 
\int_{\{ |f|\ge n_j \}} |f|\ d\mu\,,
\end{eqnarray*}
whence \eqref{b111}. 
\end{proof}

The inequality \eqref{b111} gives some clue towards the construction of a function $\Phi$ fulfilling the requirements of Theorem~\ref{thb106}. Indeed, it clearly follows from \eqref{b111} that, in order to estimate the norm of $\Phi(f)$ in $L^1(\Omega)$ uniformly with respect to $f\in \mathcal{F}$, it is sufficient to show that one can find a function $\Phi$ and a sequence $(n_k)_{k\ge 0}$ such that the sum in the right-hand side of \eqref{b111} is bounded independently of $f\in \mathcal{F}$ and $k\ge 1$. Observing that this sum is bounded from above by the series
\begin{equation}
\label{b112}
\sum_{j=0}^\infty \left({ \Phi'(n_{j+1}) - \Phi'(n_j) }\right)\ X_j
\end{equation}
with
$$
X_j := \sup_{f\in \mathcal{F}} \int_{\{ |f|\ge n_j \}} |f|\ d\mu\,, \qquad j\ge 0\,,
$$
and that $X_j\to 0$ as $j\to\infty$ by \eqref{b201}, the proof of Theorem~\ref{thb106} amounts to showing that one can find $\Phi$ and $(n_k)_{k\ge 0}$ such that the series \eqref{b112} converges.

\begin{proof}[Proof of Theorem~\ref{thb106}]
-- (i) $\Longrightarrow$ (ii). Consider two sequences of positive real numbers $(\alpha_m)_{m\ge 0}$ and $(\beta_m)_{m\ge 0}$ satisfying\begin{equation}
\label{b113}
\sum_{m=0}^\infty \alpha_m = \infty \;\;\mbox{ and }\;\; \sum_{m=0}^\infty \alpha_m \beta_m < \infty.
\end{equation}
It then follows from \eqref{b201} that there exists a non-decreasing sequence of integers $(N_m)_m\ge 0$ such that $N_0=1$, 
$N_1\ge 2$ and
\begin{eqnarray}
\label{b114}
N_{m+1}\ge \left({ 1 + \frac{\alpha_m}{\alpha_{m-1}} }\right)\ N_m, 
\;\;\; m\ge 1\,,,\\
\label{b115}
\sup_{f\in \mathcal{F}} \int_{\{ |f|\ge N_m \}} 
|f|\ d\mu \le \beta_m, \;\;\; m\ge 1\,.
\end{eqnarray}

Let us now construct the function $\Phi$ and first look for a $C^1$-smooth function which is piecewise quadratic on each  interval $[N_m,N_{m+1}]$. More precisely, we assume that, for each $m\ge 0$,
$$
\Phi'(r) = A_m\ r + B_m, \;\;\; r\in [N_m,N_{m+1}]\,, 
$$
the real numbers $A_m$ and $B_m$ being yet to be determined. In order that such a function $\Phi$ fulfills the requirements of Theorem~\ref{thb106}, $A_m$ and $B_m$ should enjoy the following properties:

\begin{itemize}
\item[\textbf{(c1)}] $(A_m)_{m\ge 0}$ is a non-increasing sequence of positive real numbers, which implies the convexity of $\Phi$ and the concavity of $\Phi'$,
\item[\textbf{(c2)}] $A_{m+1}\ N_{m+1} + B_{m+1} = A_m\ N_{m+1} + B_m$, $m\ge 0$, which ensures the continuity of $\Phi'$,
\item[\textbf{(c3)}] $A_m\ N_m + B_m \to \infty$ as $m\to\infty$, so that  \eqref{b108} is satisfied,
\item[\textbf{(c4)}] the series $\sum A_m\ \left({ N_{m+1} - N_m }\right)\ \beta_m$ converges, which, together with \eqref{b115}, ensures that the right-hand side of \eqref{b111} is bounded uniformly with respect to  $f\in \mathcal{F}$.
\end{itemize}

Let us now prove that the previously constructed sequence $(N_m)_{m\ge 0}$ allows us to find $(A_m,B_m)$ complying with the four constraints (c1)-(c4). According to \eqref{b113} and (c4), a natural choice for $A_m$ is 
$$
A_m := \frac{\alpha_m}{N_{m+1} - N_m}, \;\;\;  m\ge 0\,.
$$
The positivity of $(\alpha_m)$ and \eqref{b114} then ensure that the sequence $(A_m)_{m\ge 0}$ satisfies (c1). Next, (c2) also reads 
$$
A_{m+1}\ N_{m+1} + B_{m+1} = A_m\ N_m + B_m + \alpha_m\,, 
$$
from which we deduce that
$$
A_m\ N_m + B_m = \sum_{i=0}^{m-1} \alpha_i + A_0\ N_0 + B_0\,.
$$
The above identity allows us to determine the sequence $(B_m)_{m\ge 0}$ by
$$
B_0:=0 \;\;\mbox{ and }\;\; B_m := \sum_{i=0}^{m-1} \alpha_i + A_0\ N_0 - A_m\ N_m\,, \quad m\ge 1\,,
$$
and (c3) is a straightforward consequence of \eqref{b113}.

We are now in a position to complete the definition of $\Phi$. We set 
$$
\Phi'(r) := \left\{{
\begin{array}{cll}
\displaystyle{\frac{\alpha_0}{N_1-N_0}\ r} & \mbox{ for } & r\in [0,N_1)\,,\\
 & & \\
\displaystyle{\frac{\alpha_m\ \left({ r-N_m }\right)}{N_{m+1}-N_m} + 
  \sum_{i=0}^{m-1} \alpha_i + \frac{\alpha_0}{N_1-N_0}} & \mbox{ for } & r\in [N_m,N_{m+1})\,,\\
 & & \;m\ge 1,
\end{array}
}\right.$$
and 
$$
\Phi(r) := \int_0^r \Phi'(s)\ ds, \;\;\; r\in [0,\infty)\,.
$$

Clearly $\Phi(0)=\Phi'(0)=0$ and, for $m\ge 1$,
\begin{equation}
\label{b116}
\lim_{r\to N_m -} \Phi'(r) = \Phi'(N_m) = \sum_{i=0}^{m-1} \alpha_i 
+ \frac{\alpha_0}{N_1-N_0}.
\end{equation}
Consequently, $\Phi'\in C([0,\infty))$ and thus
$\Phi\in C^1([0,\infty))$. Moreover, $\Phi'$ is differentiable in $(0,N_1)$ and in each open interval $(N_m,N_{m+1})$ with
$$\Phi''(r) = \left\{{
\begin{array}{cll}
\displaystyle{\frac{\alpha_0}{N_1-N_0}} & \mbox{ for } & r\in (0,N_1)\,,\\
 & & \\
\displaystyle{\frac{\alpha_m}{N_{m+1}-N_m}} & \mbox{ for } & 
r\in (N_m,N_{m+1}), \;\;\; m\ge 1\,,
\end{array}
}\right.$$
and \eqref{b114} ensures that $\Phi''$ is non-negative and non-increasing, whence the convexity of $\Phi$ and the concavity of $\Phi'$. We then deduce from the monotonicity of $\Phi'$, \eqref{b113}, and \eqref{b116} that $\Phi'$ fulfills \eqref{b107} and
$$
\lim_{r\to \infty} \Phi'(r) = \infty\,.
$$ 
The property \eqref{b108} then follows by the L'Hospital rule.

We finally infer from \eqref{b111}, \eqref{b115}, and \eqref{b116} that, 
for $f\in \mathcal{F}$ and $m\ge 1$, we have
\begin{eqnarray*}
\int_{\{ |f|<N_m\}}  \Phi(|f|)\ d\mu & \le &
\Phi'(1)\ \int_\Omega |f|\ d\mu 
+ \sum_{j=0}^{m-1} \alpha_j\ \beta_j\\
& \le & \Phi'(1)\ \sup_{g\in \mathcal{F}} \int_\Omega |g|\ d\mu 
+ \sum_{j=0}^\infty \alpha_j\ \beta_j,
\end{eqnarray*}
and the right-hand side of the above inequality is finite by \eqref{b113} and the boundedness of $\mathcal{F}$ in $L^1(\Omega)$. We let $m\to \infty$ in the above inequality and conclude that $\left\{ \Phi(|f|)\ : \ f\in \mathcal{F} \right\}$ is bounded in $L^1(\Omega)$.

\medskip

We next modify the function $\Phi$ constructed above in order to improve its regularity. To this end, we define $\Phi_1\in C^1(\mathbb{R})$ by 
$$
\Phi_1(r) := \Phi(r) \;\;\mbox{ for }\;\; r\ge 0 \;\;\mbox{ and }\;\; \Phi_1(r)=\Phi''(0)\ \frac{r^2}{2} \;\;\mbox{ for }\;\; r\le 0\,.
$$
As $\Phi_1'(r)\le 0=\Phi_1'(0)$ for $r\le 0$, $\Phi_1'$ is non-decreasing so that $\Phi_1$ is convex. Similarly, $\Phi_1''(r)= \Phi''(0)$ for $r\le 0$, which guarantees that $\Phi_1''$ is non-increasing and thus the concavity of $\Phi_1'$. 

Consider next $\vartheta\in C_0^\infty(\mathbb{R})$ such that
$$
\vartheta\ge 0\,, \;\;\; \mbox{ supp }\vartheta = (-1,1)\,, \;\;\; \int_{\mathbb{R}} \vartheta(r)\ dr = 1\,.
$$ 
We define a function $\Psi$ by 
$$
\Psi(r) := (\vartheta * \Phi_1)(r) - (\vartheta * \Phi_1)(0) - (\vartheta * \Phi_1')(0)\ r\,, \;\;\; r\in \mathbb{R}\,.
$$
Clearly, $\Psi\in C^\infty(\mathbb{R})$ satisfies $\Psi(0)=\Psi'(0)=0$. Next, thanks to the non-negativity of $\vartheta$, the convexity of $\Phi_1$ and the concavity of $\Phi_1'$ imply the convexity of $\Psi$ and the concavity of $\Psi'$. Moreover, we have $\Psi'(r)>0$ for $r>0$. Indeed, assume for contradiction that $\Psi'(r_0)=0$ for some $r_0>0$. Then
$$
0 = \Psi'(r_0) = \int_{\mathbb{R}} \vartheta(s)\ \left( \Phi_1'(r_0-s) - \Phi_1'(-s) \right)\ ds\,,
$$
from which we infer that $ \vartheta(s)\ \left( \Phi_1'(r_0-s) - \Phi_1'(-s) \right) = 0$ for $s\in\mathbb{R}$ by the non-negativity of $\vartheta$ and the monotonicity of $\Phi_1'$. Taking $s=0$, we conclude that $0=\Phi_1'(r_0)=\Phi'(r_0)$, and a contradiction. Consequently, $\Psi$ fulfills \eqref{b107}. 
 
We next check that $\Psi$ is superlinear at infinity. To this end, we consider $r\ge 2$ and deduce from the monotonicity of $\Phi_1'$  and \eqref{b108} that
\begin{eqnarray*}
\Psi'(r) & = & \int_{-1}^1 \vartheta(s)\ \Phi_1'(r-s)\ ds - (\vartheta * \Phi_1')(0) \\
& \ge & \Phi_1'(r-1)\ \int_{-1}^1 \vartheta(s)\ ds - (\vartheta * \Phi_1')(0) \\
& \ge & \Phi'(r-1)\ - (\vartheta * \Phi_1')(0) \mathop{\longrightarrow}_{r\to \infty} \infty\,,
\end{eqnarray*}
and we use again the L'Hospital rule to conclude that $\Psi$ fulfills \eqref{b108}. 

Let us finally show that there is a constant $C>0$ such that
\begin{equation}
\label{b117}
\Psi(r) \le C\ (r + \Phi(r))\,, \;\;\; r\ge 0\,.
\end{equation}
Indeed, either $r> 1$ and $r-s\ge 0$ for all $s\in (-1,1)$ or $r\in [0,1]$. In the former case, as $\Phi_1$ is non-decreasing in $[0,\infty)$ and non-negative in $\mathbb{R}$, we have  
\begin{eqnarray*} 
\Psi(r) & \le & \int_{-1}^1 \vartheta(s)\ \Phi_1(r+1)\ ds - r\ \int_{-1}^1 \vartheta(s)\ \Phi_1'(-s)\ ds \\
& \le & \Phi(r+1) + \sup_{[-1,1]}{\{|\Phi_1'|\}}\ r\,.
\end{eqnarray*}
On the other hand, the concavity of $\Phi'$, the convexity of $\Phi$, and the property $\Phi(0)=\Phi'(0)=0$ entail that
\begin{align*}
\Phi(r) & = \int_0^r \Phi'\left( \frac{s}{r+1} (r+1) \right)\ ds \ge \int_0^r \frac{s}{r+1} \Phi'(r+1)\ ds \\
& = \frac{r^2}{2(r+1)^2} (r+1) \Phi'(r+1) \ge \frac{\Phi(r+1)}{4} \ .
\end{align*}
Combining the previous two estimates gives \eqref{b117} for $r\ge 1$. When $r\in [0,1]$, the convexity of $\Phi_1$ and the concavity of $\Phi_1'$ ensure that 
\begin{eqnarray*}
\Psi(r) & \le & \int_{-1}^1 \vartheta(s)\ (r-s)\ \Phi_1'(r-s)\ ds - r\ \int_{-1}^1 \vartheta(s)\ \Phi_1'(-s)\ ds \\
& \le & r\ \int_{-1}^1 \vartheta(s)\ \left( \Phi_1'(r-s) - \Phi_1'(-s) \right)\ ds \\
& \le & r\ \int_{-1}^1 \vartheta(s)\ r\ \Phi_1''(-s)\ ds \le r\ \Phi''(0)\,,
\end{eqnarray*}
whence \eqref{b117}.

Now, since $\mathcal{F}$ and $\left\{ \Phi(|f|)\ : \ f\in \mathcal{F} \right\}$ are two bounded subsets of $L^1(\Omega)$, the boundedness of $\left\{ \Psi(|f|)\ : \ f\in \mathcal{F} \right\}$ in $L^1(\Omega)$ readily follows from \eqref{b117}, which completes the proof of (i) $\Longrightarrow$ (ii) in Theorem~\ref{thb106}. 

\medskip

-- (ii) $\Longrightarrow$ (i). Let $c\in (0,\infty)$. Owing to the convexity of $\Phi$, the function $r\mapsto \Phi(r)/r$ is non-decreasing and 
\begin{eqnarray*}
\sup_{f\in \mathcal{F}} \int_{\{ |f| \ge c\}} |f|\
d\mu & = & \sup_{f\in \mathcal{F}} \int_{\{ |f| \ge c\}} 
\frac{|f|}{\Phi(|f|)}\ \Phi(|f|)\ d\mu\\
& \le & \frac{c}{\Phi(c)}\ \sup_{f\in \mathcal{F}}\ 
\int_\Omega \Phi(|f|)\ d\mu\,.
\end{eqnarray*}
It then follows from \eqref{b108} that
$$
\lim_{c\to \infty}\ \sup_{f\in \mathcal{F}}\ \int_{\{ |f|
\ge c\}} |f|\ d\mu = 0\,,
$$
whence \eqref{b201}. 
\end{proof}

\begin{remark}\label{reb118}
Notice that the sequences $(\alpha_m)$ and $(\beta_m)$ used in the proof of Theorem~\ref{thb106} can be \textit{a priori} chosen
arbitrarily provided they fulfill the condition \eqref{b113}. In particular, with the choice $\alpha_m=1$, the above construction of the function $\Phi$ is similar to that performed in \cite{Le77}.
\end{remark}

For further use, we introduce the following notation:

\begin{definition} \label{defcvp}
We define $\mathcal{C}_{VP}$ as the set of convex functions $\Phi\in C^\infty([0,\infty))$ with $\Phi(0) = \Phi'(0) = 0$ and such that $\Phi'$ is a concave function satisfying (\ref{b107}). The set $\mathcal{C}_{VP,\infty}$ denotes the subset of functions in $\mathcal{C}_{VP}$ satisfying the additional property \eqref{b108}. 
\end{definition}

A first consequence of Theorem~\ref{thb106} is that every function in $L^1(\Omega)$ enjoys an additional integrability property in the following sense. 
 
\begin{corollary} \label{cob119}
Let $f\in L^1(\Omega)$. Then there is a function 
$\Phi\in \mathcal{C}_{VP,\infty}$ such that $\Phi(|f|)\in L^1(\Omega)$.
\end{corollary}

\begin{proof} Clearly $\mathcal{F} = \{ f \}$ fulfills the  
assertion~(i) of Theorem~\ref{thb106}. 
\end{proof}

\begin{remark}\label{reb120}
If $\mu(\Omega)<\infty$, we have 
$$
\bigcup_{p>1} L^p(\Omega) \subset L^1(\Omega)\,,
$$
but this inclusion cannot be improved to an equality in general. For instance, the function $f : x\longmapsto x^{-1}\ (\ln{x})^{-2}$ belongs to $L^1(0,1/2)$ but $f\not\in L^p(0,1/2)$ as soon as $p>1$. A consequence of Corollary~\ref{cob119} is that $L^1(\Omega)$ is the union of the Orlicz spaces $L_\Phi$, see \cite{RR91} for instance.  
\end{remark}

Let us mention here that Corollary~\ref{cob119} is also established in \cite[p.~60--61]{KR61}, \cite[Proposition~A1]{MW99} and \cite{Ri96}, still without the requirement that $\Phi$ has a concave first derivative. However, the convex function $\Phi$ constructed in \cite[Proposition~A1]{MW99} and \cite{Ri96} enjoys the properties \eqref{b122} and \eqref{b124} stated below, respectively. In fact, it follows clearly from the proof of Theorem~\ref{thb106} that there is some freedom in the construction of the function $\Phi$ and this fact has allowed some authors to endow it with additional properties according to their purpose. In particular, the concavity of $\Phi'$ is useful to establish the existence of weak solutions to reaction-diffusion systems \cite{Pi87}, while the property \eqref{b122} is used to study the spatially homogeneous Boltzmann equation \cite{MW99} and the property \eqref{b124} to show the existence of solutions to the spatially inhomogeneous BGK equation \cite{Ri96}. The possibility of choosing $\Phi'$ concave is also useful in the proof of the existence of solutions to the continuous coagulation-fragmentation equation as we shall see in Section~\ref{sec:3}. We now check that all these properties are actually a consequence of the concavity of $\Phi'$.

\begin{proposition}\label{prb121}
Consider $\Phi\in \mathcal{C}_{VP}$. Then 
\begin{equation}
\label{b122}
r\mapsto \frac{\Phi(r)}{r} \;\mbox{ is concave in }\; (0,\infty)\,,
\end{equation}
\begin{equation}
\label{b123}
\Phi(r) \le r\ \Phi'(r) \le 2\ \Phi(r)\,,
\end{equation}
\begin{equation}
\label{b123b}
s \ \Phi'(r) \le \Phi(r) + \Phi(s)\,,
\end{equation}
\begin{equation}
\label{b124}
\Phi(\lambda r) \le \max{\{ 1 , \lambda^2 \}}\ \Phi(r)\,,
\end{equation}
\begin{equation}
\label{b125}
(r+s)\ \left( \Phi(r+s) - \Phi(r) - \Phi(s) \right) \le 2\ \left( r\
  \Phi(s) + s\ \Phi(r) \right)\,, 
\end{equation}
for $r\ge 0$, $s\ge 0$, and $\lambda\ge 0$.
\end{proposition} 

\begin{proof} The inequalities \eqref{b123}-\eqref{b125} being obviously true when $r=0$ or $s=0$, we consider $r>0$, $s>0$, and $t\in
[0,1]$. Thanks to the concavity of $\Phi'$, we have 
\begin{eqnarray*}
\frac{\Phi(t r + (1-t) s)}{t r + (1-t) s} & = & \int_0^1
\Phi'(z (t r + (1-t) s) )\ dz \\
& \ge & \int_0^1 \left( t\ \Phi'(z r) + (1-t)\ \Phi'(z s)
\right)\ dz \\
& \ge & t\ \frac{\Phi(r)}{r} + (1-t)\ \frac{\Phi(s)}{s}\,,
\end{eqnarray*}
whence \eqref{b122}. 

Next, the convexity of $\Phi$ ensures that 
$$
\Phi(0) - \Phi(r) \ge -r\ \Phi'(r)\,, \;\;\; r\ge 0\,,
$$
from which the first inequality in \eqref{b123} follows. Similarly, we deduce from \eqref{b122} that, for $r\ge
0$, we have
\begin{eqnarray*}
\Phi'(0) - \frac{\Phi(r)}{r} & \le & -r\ \left( \frac{\Phi'(r)}{r} -
\frac{\Phi(r)}{r^2} \right) \\
- \frac{\Phi(r)}{r} & \le & - \Phi'(r) + \frac{\Phi(r)}{r} \\
r\ \Phi'(r) & \le & 2\ \Phi(r)\,,
\end{eqnarray*}
which completes the proof of \eqref{b123}. 

Combining the convexity of $\Phi$ with \eqref{b123} gives
$$
s \Phi'(r) = (s-r) \Phi'(r) + r \Phi'(r) \le \Phi(s) -\Phi(r) + 2 \Phi(r) 
$$
for $r\ge 0$ and $s\ge 0$, hence \eqref{b123b}.

Consider now $r\ge 0$ and $\lambda\in [0,1]$. We infer from the monotonicity \eqref{b107} of $\Phi$ that
$$
\Phi(\lambda r) \le \Phi(r) \le \max{\{ 1 , \lambda^2 \}}\
\Phi(r)\,.
$$
Next, for $r\ge 0$, $s\in [0,r]$, and $\lambda>1$, it follows from the concavity and non-negativity of $\Phi'$ that 
$$
\Phi'(s) = \Phi'\left( \frac{\lambda s}{\lambda}  + \left( 1 -
\frac{1}{\lambda} \right) 0 \right) \ge \frac{\Phi'(\lambda s)}{\lambda}
\,.$$
We integrate this inequality with respect to $s$ over $(0,r)$ to
obtain
$$
\Phi(r) \ge \frac{\Phi(\lambda r)}{\lambda^2}\,,
$$
and complete the proof of \eqref{b124}.

Finally, let $r\ge 0$, $s\ge 0$, $\rho\in [0,r]$, and $\sigma\in
[0,s]$. We infer from the concavity of $\Phi'$ that
$$
\Phi'(\rho+\sigma) - \Phi'(\rho) \ge \sigma\ \Phi''(\rho+\sigma)
\;\mbox{ and }\; \Phi'(\rho+\sigma) - \Phi'(\sigma) \ge \rho\
\Phi''(\rho+\sigma)\,,
$$ 
whence
\begin{equation}
\label{b126}
(\rho+\sigma)\ \Phi''(\rho+\sigma) + 2\ \Phi'(\rho+\sigma) \le 4\
\Phi'(\rho+\sigma) - \Phi'(\rho) - \Phi'(\sigma)\,. 
\end{equation}
We use once more the concavity of $\Phi'$ to obtain
$$
\Phi''(\tau) \ge \Phi''(\tau+\sigma)\,, \;\;\; \tau\ge 0\,.
$$
Integrating this inequality with respect to $\tau$ over $(0,\rho)$ we conclude that 
\begin{equation}
\label{b127}
\Phi'(\rho+\sigma) \le \Phi'(\rho)+\Phi'(\sigma)\,,
\end{equation}
since $\Phi'(0)=0$. It next follows from \eqref{b126} and \eqref{b127}  that 
$$
(\rho+\sigma)\ \Phi''(\rho+\sigma) + 2\ \Phi'(\rho+\sigma) \le 3 \left( \Phi'(\rho) + \Phi'(\sigma) \right)\ .
$$ 
As  
\begin{align*}
& (r+s)\ \Phi(r+s) - r\ \Phi(r) - s\ \Phi(s) \\
& \qquad = \int_0^r \int_0^s \left\{ (\rho+\sigma)\ \Phi''(\rho+\sigma) + 2\ \Phi'(\rho+\sigma) \right\}\ d\sigma d\rho\ ,
\end{align*}
the previous inequality gives the upper bound
\begin{align*}
(r+s)\ \Phi(r+s) - r\ \Phi(r) - s\ \Phi(s) & \le 3\ \int_0^r \int_0^s
\left( \Phi'(\rho) + \Phi'(\sigma) \right) \ d\sigma d\rho \\
& = 3 \left( s\ \Phi(r) + r\ \Phi(s) \right)\ ,
\end{align*} 
which we combine with
\begin{align*}
(r+s) \left( \Phi(r+s) - \Phi(r) - \Phi(s) \right) & = (r+s)\ \Phi(r+s) - r\ \Phi(r) - s\ \Phi(s) \\
& \qquad - s\ \Phi(r) - r\ \Phi(s)\ ,
\end{align*}
to obtain \eqref{b125}. 
\end{proof}

\begin{remark}\label{reb128}
The property \eqref{b124} implies that $\Phi$ enjoys the so-called $\Delta_2$-condition, namely, there exists $\ell>1$ such that
$\Phi(2 r)\le \ell\ \Phi(r)$ for $r\ge 0$.  It also follows from \eqref{b123} that $\Phi$ grows at most quadratically at infinity.
\end{remark}

\subsection{Weak convergence in $L^1$ and a.e. convergence}
\label{subsec:25}

There are several connections between weak convergence in $L^1$ and almost everywhere convergence. The combination of both is actually equivalent to the strong convergence in $L^1(\Omega)$ according to Vitali's convergence theorem, see \cite[Theorem~III.3.6]{DS57} for instance.

\begin{theorem}[Vitali] \label{thd100}
Consider a sequence $(f_n)_{n\ge 1}$ in $L^1(\Omega)$ and a function $f\in L^1(\Omega)$ such that $(f_n)_{n\ge 1}$ converges $\mu$-a.e. towards $f$. The following two statements are equivalent:
\begin{itemize}
\item[\textbf{(i)}] $(f_n)_{n\ge 1}$ converges (strongly) towards $f$ in $L^1(\Omega)$.
\item[\textbf{(ii)}] The set $\left\{ f_n\ : \ n\ge 1 \right\}$ is bounded in $L^1(\Omega)$ and fulfills the conditions \eqref{b104} and \eqref{b105}. 
\end{itemize}
\end{theorem}

In other words, the weak convergence in $L^1(\Omega)$ coupled with the $\mu$-almost everywhere convergence imply the convergence in $L^1(\Omega)$.

\begin{proof} As the proof that (i) $\Longrightarrow$ (ii) is obvious, we turn to the proof of the converse and fix $\varepsilon>0$. On the one hand, we deduce from \eqref{b105} and the integrability of $f$ that there exist $\Omega_\varepsilon\in\mathcal{B}$ with $\mu(\Omega_\varepsilon)<\infty$ such that 
$$
\sup_{n\ge 1} \int_{\Omega\setminus\Omega_\varepsilon} \left( |f_n| + |f| \right)\ d\mu \le \varepsilon\,.
$$
On the other hand, since $f\in L^1(\Omega)$ and $(f_n)_{n\ge 1}$ is bounded in $L^1(\Omega)$, we have
$$
\int_{\{|f_n-f|\ge R\}} |f_n-f|\ d\mu \le \frac{1}{R}\ \left( \|f\|_1 + \sup_{m\ge 1}{\{\| f_m\|_1\}} \right) \;\;\mbox{ for }\;\; R>0\,.
$$
Then, 
\begin{eqnarray*}
\|f_n-f\|_1 & \le & \int_{\Omega\setminus\Omega_\varepsilon} \left( |f_n| + |f| \right)\ d\mu + \int_{\Omega_\varepsilon} |f_n - f|\ \mathbf{1}_{\{ |f_n - f| \le \varepsilon^{-1}\}}\ d\mu \\
& & +\ \int_{\Omega_\varepsilon} |f_n - f|\ \mathbf{1}_{\{ |f_n - f| > \varepsilon^{-1}\}}\ d\mu \\
& \le & \varepsilon \left( 1 + \|f\|_1 + \sup_{m\ge 1}{\{\| f_m\|_1\}} \right) + \int_{\Omega_\varepsilon} |f_n - f|\ \mathbf{1}_{\{ |f_n - f| \le \varepsilon^{-1}\}}\ d\mu\,.
\end{eqnarray*}
Since $\Omega_\varepsilon$ has a finite measure, we now infer from the almost everywhere convergence of $(f_n)_{n\ge 1}$ and the Lebesgue dominated convergence theorem that the last term of the right-hand side of the above inequality converges to zero as $n\to\infty$. Consequently, 
$$
\limsup_{n\to\infty} \|f_n-f\|_1 \le \varepsilon \left( 1 + \|f\|_1 + \sup_{m\ge 1}{\{\| f_m\|_1\}} \right)\,.
$$
Letting $\varepsilon\to 0$ completes the proof. 
\end{proof}

\begin{remark}\label{red101} The $\mu$-a.e. convergence of $(f_n)_{n\ge 1}$ in Theorem~\ref{thd100} can be replaced by the convergence in measure.
\end{remark}

Another useful consequence is the following result which is implicitly used in \cite{DPL89,St89}, for instance. It allows one to identify the limit of the product of a weakly convergent sequence in $L^1$ with a bounded sequence which has an almost everywhere limit.

\begin{proposition}\label{prd2}
Let $(f_n)_{n\ge 1}$ be a sequence of measurable functions in $L^1(\Omega)$ and $(g_n)_{n\ge 1}$ be a sequence of measurable functions in $L^\infty(\Omega)$. Assume further that there are $f\in L^1(\Omega)$ and $g\in L^\infty(\Omega)$ such that
\begin{equation}
\label{d1}
f_n \rightharpoonup f \;\mbox{ in }\; L^1(\Omega)\,,
\end{equation}
\begin{equation}
\label{d2}
\vert g_n(x)\vert \le M \;\mbox{ and }\; \lim_{n\to \infty} g_n(x)
= g(x) \;\;\mu-\mbox{a.e. }
\end{equation} 
Then 
\begin{equation}
\label{d3}
\lim_{n\to \infty} \int_\Omega \vert f_n\vert\ \vert g_n - g\vert\ d\mu = 0
\;\mbox{ and }\; f_n g_n \rightharpoonup f g \;\mbox{ in }\; L^1(\Omega)\,.
\end{equation}
\end{proposition}

The proof of Proposition~\ref{prd2} combines the Dunford-Pettis theorem (Theorem~\ref{thb103}) with Egorov's theorem which we recall now, see \cite[p.~73]{Ru87} for instance.

\begin{theorem}[Egorov]\label{thb10000}
Assume that $\mu(\Omega)<\infty$ and consider a sequence $(h_n)_{n\ge 1}$ of measurable functions in $\Omega$ such that $h_n \to h $ $\mu-$a.e. for some measurable function $h$. Then, for any $\delta>0$, there is a measurable subset $A_\delta\in\mathcal{B}$ such that
$$
\mu(A_\delta) \le \delta \;\;\;\text{ and }\;\;\; \lim_{n\to\infty} \sup_{x\in\Omega\setminus A_\delta} |h_n(x)-h(x)| = 0\ .
$$  
\end{theorem}

\begin{proof}[Proof of Proposition~\ref{prd2}] Let $\varepsilon\in (0,1)$. On the one hand, the Dunford-Pettis theorem and \eqref{d1} ensure that there exist
$\delta>0$ and  $\Omega_\varepsilon\subset \Omega$ such that $\mu(\Omega_\varepsilon)<\infty$,  
$$ 
\sup_{n\ge 1} \int_{\Omega\setminus\Omega_\varepsilon} |f_n|\ d\mu \le \frac{\varepsilon}{4 M}\,, \;\;\mbox{ and }\;\; 
\eta\left\{ (f_n)_{n\ge 1},\delta \right\} \le \frac{\varepsilon}{4 M}\,.
$$
On the other hand, since $\mu(\Omega_\varepsilon)<\infty$, we deduce from Egorov's theorem and \eqref{d2} that there is $\mathcal{O}_\varepsilon\subset \Omega_\varepsilon$ such that 
$$
\mu(\Omega_\varepsilon\setminus \mathcal{O}_\varepsilon)\le\delta \;\;\mbox{ and }\;\; \lim_{n\to \infty} \sup_{x\in \mathcal{O}_\varepsilon} |(g_n-g)(x)| = 0\,.
$$ 
Then
\begin{eqnarray*}
\int_\Omega |f_n|\ |g_n - g|\ d\mu & \le &
2M\ \int_{\Omega\setminus\Omega_\varepsilon} |f_n|\ d\mu +
2M\ \int_{\Omega_\varepsilon\setminus\mathcal{O}_\varepsilon} |f_n|\ d\mu \\
& & +\ \int_{\mathcal{O}_\varepsilon} |f_n|\ |g_n - g|\ d\mu \\
& \le & \varepsilon + \sup_{m\ge 1} \| f_m\|_{L^1}\
\sup_{x\in\mathcal{O}_\varepsilon} |(g_n-g)(x)|\,.
\end{eqnarray*}
Consequently, 
$$
\limsup_{n\to \infty} \int_\Omega |f_n|\ |g_n - g|\ d\mu \le \varepsilon\,,
$$
and $(f_n(g_n-g))_{n\ge 1}$ converges strongly towards zero in
$L^1(\Omega)$. Since $g\in L^\infty(\Omega)$, the second statement in Proposition~\ref{prd2} readily follows from the first one and \eqref{d1}. 
\end{proof}

\begin{remark}\label{red10001} 
Proposition~\ref{prd2} is somehow an extension of the following classical result: Let $p\in (1,\infty)$. If $f_n\rightharpoonup f$ in $L^p(\Omega)$ and $g_n \longrightarrow g$ in $L^{p/(p-1)}(\Omega)$, then $f_n g_n \rightharpoonup fg$ in $L^1(\Omega)$.
\end{remark}

The final result of this section is a generalization of  Proposition~\ref{prd2} and Remark~\ref{red10001}.

\begin{proposition}\label{prd3}
Let $\psi\in C([0,\infty)$ be a non-negative convex function satisfying $\psi(0)=0$ and $\psi(r)\ge C_0 r$ for $r\ge 1$ and some $C_0>0$, and denote its convex conjugate function by $\psi^*$. Assume that $\mu(\Omega)<\infty$ and consider two sequences $(f_n)_{n\ge 1}$ and $(g_n)_{n\ge 1}$ of real-valued integrable functions in $\Omega$ enjoying the following properties: there are $f$ and $g$ in $L^1(\Omega)$ such that
\begin{enumerate}
\item $f_n \rightharpoonup f$ in $L^1(\Omega)$ and 
$$
C_1 := \sup_{n\ge 1} \int_\Omega \psi(|f_n|)\ d\mu < \infty\ ,
$$
\item $g_n \longrightarrow g$ $\mu-$a.e. in $\Omega$,
\item for each $\varepsilon\in (0,1]$, the family $\mathcal{G}_\varepsilon := \left\{ \psi^*(|g_n|/\varepsilon)\ :\ n\ge 1 \right\}$ is uniformly integrable in $L^1(\Omega)$.
\end{enumerate}
Then
$$
f_n g_n \rightharpoonup f g \;\;\text{ in}\;\; L^1(\Omega)\ .
$$
\end{proposition}
 
\begin{proof} We first recall that, given $\varepsilon\in (0,1]$, the uniform integrability of $\mathcal{G}_\varepsilon$ in $L^1(\Omega)$ ensures that
\begin{equation}
C_2(\varepsilon) := \sup_{n\ge 1} \int_\Omega \psi^*\left( \frac{|g_n|}{\varepsilon} \right)\ d\mu < \infty\ , \label{tl100}
\end{equation}
and
\begin{equation}
\lim_{\delta\to 0} \eta\{\mathcal{G}_\varepsilon,\delta\} = 0\ , \label{tl101}
\end{equation}
the modulus of uniform integrability $\eta$ being defined in Definition~\ref{deb100}.
We next observe that, thanks to Young's inequality 
\begin{equation}
rs \le \psi(r) + \psi^*(s)\ , \qquad (r,s)\in [0,\infty)^2\ , \label{tl6}
\end{equation}
which ensures, together with \eqref{tl100} (with $\varepsilon=1$), that
$$
\int_\Omega |f_n g_n|\ d\mu \le \int_\Omega \left( \psi(|f_n|) + \psi^*(|g_n|) \right)\ d\mu \le C_1 + C_2(1)\ .
$$
Consequently, $(f_n g_n)_{n\ge 1}$ is a bounded sequence in $L^1(\Omega)$. Furthermore, the convexity of $\psi$, the weak convergence of $(f_n)_{n\ge 1}$, and a weak lower semicontinuity argument entail that 
\begin{equation}
\psi(|f|) \in L^1(\Omega) \;\;\text{ with }\;\; \int_\Omega \psi(|f|)\ d\mu \le C_1\ , \label{tl103}
\end{equation}
while the $\mu-$almost everywhere convergence of $(g_n)_{n\ge 1}$ along with \eqref{tl100} and the Fatou lemma ensure that, for each $\varepsilon\in (0,1]$ and $\delta>0$,
\begin{equation}
\int_\Omega \psi^*\left( \frac{|g|}{\varepsilon} \right)\ d\mu \le C_2(\varepsilon) \;\;\text{ and }\;\;  \eta\left\{ \left\{ \psi^*\left( \frac{|g|}{\varepsilon} \right) \right\},\delta \right\} \le \eta\{\mathcal{G}_\varepsilon,\delta\}\ . \label{tl105}
\end{equation}
In particular, $f g \in L^1(\Omega)$ as a consequence of \eqref{tl6}, \eqref{tl103}, and \eqref{tl105}. 

We now fix $\varepsilon\in (0,1]$ and $\delta\in (0,1)$. On the one hand, since $\mu(\Omega)<\infty$, we infer from Egorov's theorem that there is a measurable subset $A_{\delta}$ of $\Omega$ such that
\begin{equation}
\mu(A_{\delta})\le\delta \;\;\text{ and }\;\; \lim_{n\to\infty} \sup_{x\in \Omega\setminus A_{\delta}}{|g_n(x)-g(x)|} = 0\ . \label{tl7}
\end{equation}
On the other hand, since $g\in L^1(\Omega)$, there exists $k_\delta\ge 1$ such that 
\begin{equation}
\mu\left( \left\{ x\in \Omega\ :\ |g(x)|\ge k_\delta \right\} \right) \le \delta\ , \label{tl9}
\end{equation}
and we define $g_{\delta} := g\, \mathbf{1}_{(-k_\delta,k_\delta)}(g)$.

Now, for $\xi\in L^\infty(\Omega)$, we define
\begin{equation*}
I(n) := \int_{\Omega} (f_n g_n - f g) \xi\ d\mu\ , 
\end{equation*}
which we estimate as follows:
\begin{align*}
|I(n,R)| & \le \left| \int_{\Omega} (f_n-f) g \xi\ d\mu \right| + \left| \int_{\Omega} f_n (g_n-g) \xi\ d\mu \right| \\
& \le \left| \int_{\Omega} (f_n-f) g_{\delta} \xi\ d\mu \right| + \int_{\Omega} (|f_n| + |f|) |g-g_{\delta}| |\xi|\ d\mu \\
& \qquad + \int_{\Omega\setminus A_{\delta}} |f_n| |g_n-g| |\xi|\ d\mu + \int_{A_{\delta}} |f_n| (|g_n|+|g|) |\xi|\ d\mu\ .
\end{align*}
It first follows from Young's inequality \eqref{tl6} and the convexity of $\psi$ that
\begin{align*}
I_1(n,\delta) & := \int_{\Omega} (|f_n| + |f|) |g-g_{\delta}| |\xi|\ d\mu \le \int_{\Omega\cap \{|g|\ge k_\delta\}} (|f_n| + |f|) |g| |\xi|\ d\mu\\
& \le \|\xi\|_{\infty} \int_{\Omega\cap \{|g|\ge k_\delta\}} \left[ \psi(\varepsilon |f_n|) + \psi(\varepsilon |f|) + 2 \psi^*\left( \frac{|g|}{\varepsilon} \right) \right]\ d\mu \\
& \le \|\xi\|_{\infty} \left[ \varepsilon \int_\Omega \left( \psi(|f_n|) + \psi(|f|) \right)\ d\mu + 2 \int_{\Omega\cap \{|g|\ge k_\delta\}} \psi^*\left( \frac{|g|}{\varepsilon} \right)\ d\mu \right]\ .
\end{align*}
We then infer from \eqref{tl103}, \eqref{tl105}, and \eqref{tl9} that
\begin{equation}
I_1(n,\delta) \le 2 \|\xi\|_{\infty} \left( \varepsilon C_1 + \eta\{\mathcal{G}_\varepsilon,\delta\} \right)\ . \label{tl12}
\end{equation}
Next,
\begin{equation} 
I_2(n,\delta) := \int_{\Omega\setminus A_{\delta}} |f_n| |g_n-g| |\xi|\ d\mu \le \|\xi\|_{\infty} \sup_{\Omega\setminus A_{\delta}}{\{|g_n-g|\}}\ \sup_{m\ge 1}{\{\|f_m\|_{1}\}}\ . \label{tl13}
\end{equation}
We finally infer from \eqref{tl6}, \eqref{tl105}, \eqref{tl7}, and the convexity of $\psi$ that
\begin{align}
I_3(n,\delta) & := \int_{A_{\delta}} |f_n| (|g_n|+|g|) |\xi|\ d\mu \nonumber \\
& \le \|\xi\|_{\infty} \int_{A_{\delta}} \left[ 2 \psi(\varepsilon |f_n|) + \psi^*\left( \frac{|g_n|}{\varepsilon} \right) + \psi^*\left( \frac{|g|}{\varepsilon} \right) \right]\ d\mu \nonumber \\
& \le 2 \|\xi\|_{\infty} \left[ \varepsilon \int_{\Omega}  \psi(|f_n|)\ d\mu + \eta\{\mathcal{G}_\varepsilon,\delta\} \right] \nonumber \\
& \le 2 \|\xi\|_{\infty} \left( \varepsilon C_1 + \eta\{\mathcal{G}_\varepsilon,\delta\} \right)\ . \label{tl14}
\end{align}
Combining \eqref{tl12}-\eqref{tl14} we end up with
\begin{align}
|I(n)| & \le \left| \int_{\Omega} (f_n-f) g_{\delta} \xi\ d\mu \right| +  \|\xi\|_{\infty} \sup_{\Omega\setminus A_{\delta}}{\{|g_n-g|\}}\ \sup_{m\ge 1}{\{\|f_m\|_1\}} \nonumber \\ 
& \qquad + 4 \|\xi\|_{\infty} \left( \varepsilon C_1 + \eta\{\mathcal{G}_\varepsilon,\delta\} \right)\ . \label{tl11}
\end{align}

Now, we first let $n\to\infty$ in the above inequality and use the weak convergence of $(f_n)_{n\ge 1}$ in $L^1(\Omega)$, the boundedness of $g_{\delta}$, and the uniform convergence \eqref{tl7} to obtain
$$
\limsup_{n\to\infty} \left| \int_{\Omega} (f_n g_n - f g) \xi\ d\mu \right| \le 4 \|\xi\|_{\infty} \left( \varepsilon C_1 + \eta\{\mathcal{G}_\varepsilon,\delta\} \right)\ .
$$
We next use the uniform integrability \eqref{tl101} to pass to the limit as $\delta\to 0$ in the above estimate and find
$$
\limsup_{n\to\infty} \left| \int_{\Omega} (f_n g_n - f g) \xi\ d\mu \right| \le  4 \varepsilon C_1 \|\xi\|_{\infty} \ .
$$
We finally let $\varepsilon\to 0$ to complete the proof. 
\end{proof}

\section{Smoluchowski's coagulation equation}
\label{sec:3}

We now turn to Smoluchowski's coagulation equation 
\begin{eqnarray}
\partial_t f(t,x) & = & \frac{1}{2} \int_0^x K(y,x-y)\ f(t,y)\ f(t,x-y)\ dy \nonumber\\
& & \ - \int_0^\infty K(x,y)\ f(t,x)\ f(t,y)\ dy\,, \quad (t,x)\in (0,\infty)\times (0,\infty)\,, \label{sc1} \\
f(0,x) & = & f^{in}(x)\,, \quad x\in (0,\infty)\,, \label{sc2}
\end{eqnarray}
and collect and derive several properties of its solutions in the next sections. As outlined in the introduction, some of these properties depend heavily on the growth of the coagulation kernel $K$ which is a non-negative and symmetric function. For further use, we introduce the following notation: For $\mu\in\mathbb{R}$, the space of integrable functions with a finite moment of order $\mu$ is denoted by
\begin{equation}
L_\mu^1(0,\infty) := \left\{ g\in L^1(0,\infty)\ : \ \|g\|_{1,\mu} := \int_0^\infty (1+x^\mu) |g(x)|\ dx < \infty \right\}\,, \label{sc3}
\end{equation}
and we define
\begin{equation*}
M_\mu(g) := \int_0^\infty x^\mu g(x)\ dx\,, \quad g\in L_\mu^1(0,\infty)\ . 
\end{equation*}
Note that $L_0^1(0,\infty)=L^1(0,\infty)$ and $\|\cdot\|_{1,0} = \|\cdot\|_1$. Next, for a measurable function $g$ and $x>0$, we set
\begin{eqnarray*}
Q_1(g)(x) & := & \frac{1}{2} \int_0^x K(y,x-y)\ g(y)\ g(x-y)\ dy\ , \\
L(g)(x) & := & \int_0^\infty K(x,y)\ g(y)\ dy\ ,  \qquad Q_2(g)(x) := g(x) L(g(x))\ ,
\end{eqnarray*}
whenever it makes sense. 

\subsection{Existence: Bounded kernels}
\label{sec:31}

The first step towards the existence of solutions to \eqref{sc1}-\eqref{sc2} is to handle the case of bounded coagulation kernels. 

\begin{proposition}\label{pr.sc1}
If there is $\kappa_0>0$ such that 
\begin{equation}
0 \le K(x,y)=K(y,x) \le \kappa_0\,, \quad (x,y)\in (0,\infty)\times (0,\infty)\,, \label{sc7}
\end{equation}
and
\begin{equation}
f^{in} \in L^1(0,\infty)\,, \quad f^{in}\ge 0 \;\text{ a.e. in }\; (0,\infty)\,,\label{sc8}
\end{equation}
then there is a unique global solution $f\in C^1([0,\infty);L^1(0,\infty))$ to \eqref{sc1}-\eqref{sc2} such that 
\begin{equation}
f(t,x)\ge 0 \;\;\text{ for a.e. }\;\; x\in (0,\infty) \;\;\text{ and }\;\; \|f(t)\|_1 \le \|f^{in}\|_1\,, \quad t\ge 0\,. \label{sc8a}
\end{equation}
Furthermore, if $f^{in}\in L_1^1(0,\infty)$, then
\begin{equation}
f(t)\in L_1^1(0,\infty) \;\;\text{ and }\;\; M_1(f(t)) = M_1(f^{in})\,, \quad t\ge 0\,. \label{sc9}
\end{equation}
\end{proposition}

\begin{proof} \textbf{Step~1.} We first consider an initial condition $f^{in}$ satisfying \eqref{sc8} and prove the first statement of Proposition~\ref{pr.sc1}. We note that $Q_1$ and $Q_2$ are locally Lipschitz continuous from $L^1(0,\infty)$ to $L^1(0,\infty)$ with
$$
\|Q_i(f)-Q_i(g)\|_{1} \le \kappa_0 \left( \|f\|_1 + \|g\|_1 \right) \|f-g\|_1
$$
for $(f,g)\in L^1(0,\infty)\times L^1(0,\infty)$ and $i=1,2$. Then, denoting the positive part of a real number $r$ by $r_+:=\max{\{r , 0\}}$, the map $f\mapsto Q_1(f)_+$ is also locally Lipschitz continuous from $L^1(0,\infty)$ to $L^1(0,\infty)$ and it follows from classical results on the well-posedness of differential equations in Banach spaces (see \cite[Theorem~7.6]{Am90} for instance) that there is a unique solution $f\in C^1([0,T_m);L^1(0,\infty))$ defined on the maximal time interval $[0,T_m)$ to the differential equation
\begin{equation}
\frac{df}{dt} = Q_1(f)_+ - Q_2(f)\,, \quad t\in (0,T_m)\,, \label{sc10}
\end{equation}
with initial condition $f(0)=f^{in}$. Since the positive part is a Lipschitz continuous function and $f\in C^1([0,T_m);L^1(0,\infty))$, the chain rule gives 
$$
\partial_t (-f)_+ = - \mathrm{sign}_+(-f)\ \partial_t f\,,
$$
where $\mathrm{sign}_+(r)=1$ for $r\ge 0$ and $\mathrm{sign}_+(r)=0$ for $r<0$. We then infer from \eqref{sc10} that
$$
\partial_t (-f)_+= - \mathrm{sign}_+(-f)\ Q_1(f)_+ + \mathrm{sign}_+(-f)\ Q_2(f) \le (-f)_+\ L(f)
$$
and thus
$$
\frac{d}{dt} \| (-f)_+\|_1 \le \int_0^\infty (-f)_+\ L(f) \le \kappa_0 \|f\|_1 \| (-f)_+\|_1\,.
$$
Since $(-f)_+(0)=(-f^{in})_+=0$, we readily deduce that $(-f)_+(t)=0$ for all $t\in [0,T_m)$, that is, $f(t)\ge 0$ a.e. in $(0,\infty)$. Consequently, $Q_1(f)_+=Q_1(f)$ and it follows from \eqref{sc10} that $f$ is a solution to \eqref{sc1}-\eqref{sc2} defined for $t\in [0,T_m)$. To show that $T_m=\infty$, it suffices to notice that, thanks to the just established non-negativity of $f$, Fubini's theorem gives
\begin{align*}
\frac{d}{dt} \| f(t)\|_1 & = \int_0^\infty [ Q_1(f)(t,x) - Q_2(f)(t,x) ]\ dx \\
& = - \frac{1}{2} \int_0^\infty \int_0^\infty K(x,y) f(t,x) f(t,y)\ dydx \le 0\,,
\end{align*}
for $t\in [0,T_m)$, which prevents the blowup in finite time of the $L^1$-norm of $f$ and thereby guarantees that $T_m=\infty$. 

\noindent\textbf{Step~2.} A straightforward consequence of Fubini's theorem is the following identity for any $\vartheta\in L^\infty(0,\infty)$:
\begin{equation}
\frac{d}{dt} \int_0^\infty \vartheta(x) f(t,x)\ dx = \frac{1}{2} \int_0^\infty \int_0^\infty \tilde{\vartheta}(x,y) K(x,y) f(t,x) f(t,y)\ dydx \label{sc11}
\end{equation}
where
\begin{equation}
\tilde{\vartheta}(x,y) := \vartheta(x+y) -\vartheta(x) - \vartheta(y)\,, \quad (x,y)\in (0,\infty)\times (0,\infty)\,. \label{sc12}
\end{equation}
As a consequence of \eqref{sc11} (with $\vartheta\equiv 1$), we recover the already observed monotonicity of $t\mapsto M_0(f(t))$ and complete the proof of \eqref{sc8a}. 

\noindent\textbf{Step~3.} We now turn to an initial condition $f^{in}$ having a finite first moment and aim at proving \eqref{sc9}. Formally, \eqref{sc9} follows from \eqref{sc11} with the choice $\vartheta(x)=x$ since $\tilde{\vartheta}\equiv 0$ in that case. However, $\mathrm{id}: x\mapsto x$ does not belong to $L^\infty(0,\infty)$ and an approximation argument is required to justify \eqref{sc9}. More precisely, given $A>0$, define $\vartheta_A(x) := \min\{x,A\}$ for $x>0$. The corresponding function $\tilde{\vartheta}_A$ given by \eqref{sc12} satisfies
\begin{equation}
\tilde{\vartheta}_A(x,y) = \left\{
\begin{array}{lcl}
0 & \mbox{ if } & 0 \le x + y \le A\,, \\
A - x - y & \mbox{ if } & 0 \le \max\{x,y\} \le A < x+y \,, \\
- \min\{x,y\}  & \mbox{ if } & 0 \le \min\{x,y\} \le A < \max\{x,y\}  < x+y \,, \\
- A & \mbox{ if } & A \le \min\{x,y\} \,.
\end{array}
\right. \label{sc13}
\end{equation}
In particular $\tilde{\vartheta}_A \le 0$ and it follows from \eqref{sc11} that
$$
\int_0^\infty \vartheta_A(x) f(t,x)\ dx \le \int_0^\infty \vartheta_A(x) f^{in}(x)\ dx \le M_1(f^{in})\,, \quad t\ge 0\,.
$$
Since $\vartheta_A\to \mathrm{id}$ as $A\to\infty$, the Fatou lemma entails that 
\begin{equation}
M_1(f(t))\le M_1(f^{in})\,, \quad t\ge 0\,, \label{sc14}
\end{equation} 
and thus that $f(t)\in L_1^1(0,\infty)$ for all $t\ge 0$. To prove the conservation of matter, we use again \eqref{sc11} with $\vartheta=\vartheta_A$ and find
\begin{align}
\int_0^\infty \vartheta_A(x) f(t,x)\ dx & - \int_0^\infty \vartheta_A(x) f^{in}(x)\ dx \nonumber \\
& = - \int_0^t \left( I_1(s,A) + I_2(s,A) + I_3(s,A) \right)\ dx \label{sc15}  
\end{align}
with
\begin{align*}
I_1(s,A) & := \frac{1}{2} \int_0^A \int_{A-x}^A (x+y-A) K(x,y) f(s,x) f(s,y)\ dydx\,, \\
I_2(s,A) & := \int_0^A \int_{A}^\infty x K(x,y) f(s,x) f(s,y)\ dydx\,, \\
I_3(s,A) & := \frac{A}{2} \int_A^\infty \int_A^\infty K(x,y) f(s,x) f(s,y)\ dydx\,.
\end{align*}
On the one hand, it readily follows from \eqref{sc7} and \eqref{sc14} that
\begin{equation}
I_2(s,A) + I_3(s,A) \le \frac{\kappa_0}{A} M_1(f(s))^2 \le  \frac{\kappa_0}{A} M_1(f^{in})^2\,. \label{sc16}
\end{equation}
On the other hand, by \eqref{sc7}, 
\begin{align*}
0 \le I_1(s,A) & \le \frac{\kappa_0}{2} \int_0^A \int_{A-x}^A y f(s,x) f(s,y)\ dydx \\ 
& \le \frac{\kappa_0}{2} \int_0^\infty \int_0^\infty \mathbf{1}_{(0,A)}(x) \mathbf{1}_{(A,\infty)}(x+y) y f(s,x) f(s,y)\ dydx\,.
\end{align*}
Owing to \eqref{sc8a} and \eqref{sc14}, the Lebesgue dominated convergence theorem guarantees that
\begin{equation}
\lim_{A\to\infty} \int_0^t I_1(s,A)\ ds = 0\,. \label{sc17} 
\end{equation}
Thanks to \eqref{sc16} and \eqref{sc17} we may pass to the limit as $A\to\infty$ in \eqref{sc15} and conclude that $M_1(f(t))=M_1(f^{in})$ for $t\ge 0$. This completes the proof. 
\end{proof}

\begin{remark}\label{re.sc1}
Another formal consequence of \eqref{sc11} is that, whenever it makes sense, $t\mapsto M_\mu(f(t))$ is non-increasing for $\mu\in (-\infty,1)$ and non-decreasing for $\mu\in (1,\infty)$. Similarly, 
$$
t \mapsto \int_0^\infty \left( e^{\alpha x} - 1 \right) f(t,x)\ dx \;\;\text{ is non-decreasing for }\;\; \alpha>0\,.
$$
\end{remark}

\subsection{Existence: Unbounded kernels}
\label{sec:32}

As already mentioned, most of the coagulation rates encountered in the literature are unbounded and grow without bound as $(x,y)\to\infty$ or as $(x,y)\to (0,0)$. In that case, there does not seem to be a functional framework in which $Q_1$ and $Q_2$ are locally Lipschitz continuous and implementing a fixed point procedure does not seem to be straightforward. A different approach is then required and we turn to a compactness method which can be summarized as follows:
\begin{enumerate}
\item Build a sequence of approximations of the original problem which depends on a parameter $n\ge 1$, for which the existence of a solution is simple to show, and which converges in some sense to the original problem as $n\to \infty$.
\item Derive estimates which are independent of $n\ge 1$ and guarantee the compactness with respect to the size variable $x$ and the time variable $t$ of the sequence of solutions to the approximations.
\item Show convergence as $n\to\infty$.
\end{enumerate}

To be more precise, let $K$ be a non-negative and symmetric locally bounded function and consider an initial condition 
\begin{equation}
f^{in} \in L_1^1(0,\infty)\,, \quad f^{in}\ge 0 \;\text{ a.e. in }\; (0,\infty)\,.\label{sc18}
\end{equation}
Given an integer $n\ge 1$, a natural approximation is to truncate the coagulation kernel $K$ and define
\begin{equation}
K_n(x,y) := \min\{ K(x,y) , n \}\,, \quad (x,y)\in (0,\infty)\times (0,\infty)\,. \label{sc19}
\end{equation}
Clearly, $K_n$ is a non-negative, bounded, and symmetric function and we infer from \eqref{sc18} and Proposition~\ref{pr.sc1} that the initial-value problem \eqref{sc1}-\eqref{sc2} with $K_n$ instead of $K$ has a unique non-negative solution $f_n\in C^1([0,\infty);L^1(0,\infty))$ which satisfies
\begin{equation}
M_0(f_n(t)) \le M_0(f^{in}) \;\;\mbox{ and }\;\; M_1(f_n(t)) = M_1(f^{in})\,, \quad t\ge 0\,. \label{sc20}
\end{equation}
The compactness properties provided by the previous estimates are rather weak and the next step is to identify an appropriate topology for the compactness approach to work. A key observation in that direction is that, though being nonlinear, equation \eqref{sc1} is a nonlocal quadratic equation, in the sense that it does not involve nonlinearities of the form $f(t,x)^2$ but of the form $f(t,x) f(t,y)$ with $x\ne y$. While the former requires convergence in a strong topology to pass to the limit, the latter complies well with weak topologies. As first noticed in \cite{St89} the weak topology of $L^1$ turns out to be a particularly well-suited framework to prove the existence of solutions to \eqref{sc1}-\eqref{sc2} for several classes of unbounded coagulation kernels. In the remainder of this section, we will show how to use the tools described in Section~\ref{sec:2} to achieve this goal.  

\subsubsection{Sublinear kernels}
\label{sec:321}

We first consider the case of locally bounded coagulation kernels with a sublinear growth at infinity. More precisely, we assume that there is $\kappa>0$ such that
\begin{eqnarray}
0 & \le & K(x,y) = K(y,x) \le \kappa (1+x) (1+y)\,, \quad (x,y)\in (0,\infty)\times (0,\infty)\,, \label{sc21} \\
\omega_R(y) & := & \sup_{x\in (0,R)} \frac{K(x,y)}{y} \mathop{\longrightarrow}_{y\to\infty} 0\,. \label{sc22}
\end{eqnarray}

The following existence result is then available, see \cite{LM02, LT81, No99, Sp84}.

\begin{theorem}\label{th.sc1}
Assume that the coagulation kernel $K$ satisfies \eqref{sc21}-\eqref{sc22} and consider an initial condition $f^{in}$ satisfying \eqref{sc18}. There is a non-negative function $$
f\in C([0,\infty);L^1(0,\infty)) \cap L^\infty(0,\infty;L_1^1(0,\infty))
$$ 
such that 
\begin{align}
& \int_0^\infty \vartheta(x) \left( f(t,x) - f^{in}(x) \right)\ dx \nonumber \\ 
& \qquad =  \frac{1}{2} \int_0^t \int_0^\infty \int_0^\infty \tilde{\vartheta}(x,y) K(x,y) f(s,x) f(s,y)\ dydxds \label{sc23}
\end{align}
for all $t>0$ and $\vartheta\in L^\infty(0,\infty)$ (with $\tilde{\vartheta}$ given by \eqref{sc12}) and
\begin{equation}
M_0(f(t)) \le M_0(f^{in}) \;\;\mbox{ and }\;\; M_1(f(t)) \le M_1(f^{in})\,, \quad t\ge 0\,. \label{sc24}
\end{equation}
\end{theorem}

On the one hand, Theorem~\ref{th.sc1} excludes the two important (and borderline) cases $K_1(x,y)=x+y$ and $K_2(x,y)=xy$ which will be handled in Section~\ref{sec:322} and Section~\ref{sec:323}, respectively. On the other hand, owing to the possible occurrence of the gelation phenomenon already mentioned in the introduction, it is not possible to improve the second inequality in \eqref{sc24} to an equality in general. 

We now turn to the proof of Theorem~\ref{th.sc1}: for $n\ge 1$ define $K_n$ by \eqref{sc19} and let 
$$
f_n\in C^1([0,\infty);L^1(0,\infty))
$$ 
be the non-negative solution to \eqref{sc1}-\eqref{sc2} with $K_n$ instead of $K$ which satisfies \eqref{sc20}. To prove the weak compactness in $L^1(0,\infty)$ of $(f_n(t))_{n\ge 1}$ for each $t\ge 0$, we aim at using the Dunford-Pettis theorem (Theorem~\ref{thb103}). To this end, we shall study the behaviour of $(f_n(t))_{n\ge 1}$ on sets with small measure and for large values of $x$. Owing to \eqref{sc22}, we shall see below that the boundedness of $(M_1(f_n(t)))_{n\ge 1}$ guaranteed by \eqref{sc20} is sufficient to control the behaviour for large $x$. We are left with the behaviour on sets with small measure which we analyze in the next lemma.

\begin{lemma}\label{le.sc1}
Let $\Psi\in\mathcal{C}_{VP}$ be such that $\Psi(f^{in})\in L^1(0,\infty)$, the set $\mathcal{C}_{VP}$ being defined in Definition~\ref{defcvp}. For each $R>0$, there is $C_1(R)>0$ depending only on $K$, $f^{in}$, and $R$ such that
\begin{equation}
\int_0^R \Psi(f_n(t,x))\ dx \le \left( \int_0^R \Psi(f^{in}(x))\ dx \right) e^{C_1(R)t} \,, \quad t\ge 0\,, \quad n\ge 1\,. \label{sc25}
\end{equation}
\end{lemma}

\begin{proof} Fix $R>0$. Since $\Psi'$, $K_n$, and $f_n$ are non-negative functions and $K_n\le K$, we infer from \eqref{sc1} and Fubini's theorem that 
\begin{align*}
\frac{d}{dt} \int_0^R \Psi(f_n(t,x))\ dx & \le \frac{1}{2}\ \int_0^R \int_0^x K(x-y,y) f_n(t,x-y) f_n(t,y) \ dy\ \Psi'(f_n(t,x))\ dx \\
& \le \int_0^R \int_y^R K(x-y,y) f_n(t,x-y)\ \Psi'(f_n(t,x))\ dx\ f_n(t,y) \ dy \,.
\end{align*}
Since $\Psi\in\mathcal{C}_{VP}$ we deduce from \eqref{b123b}, \eqref{sc20}, and \eqref{sc21} that
\begin{align*}
\frac{d}{dt} \int_0^R \Psi(f_n(t,x))\ dx & \le \int_0^R \int_y^R K(x-y,y) \Psi(f_n(t,x-y))\ dx\ f_n(t,y) \ dy \\
& \qquad + \int_0^R \int_y^R K(x-y,y) \Psi(f_n(t,x))\ dx\ f_n(t,y) \ dy \\
& \le 2 \kappa (1+R)^2\ M_0(f_n)\ \int_0^R \Psi(f_n(t,x))\ dx \\
& \le 2 \kappa (1+R)^2\ M_0(f^{in})\ \int_0^R \Psi(f_n(t,x))\ dx\,.
\end{align*}
Setting $C_1(R) := 2 \kappa (1+R)^2\ M_0(f^{in})$ we obtain \eqref{sc25} after integration with respect to time. 
\end{proof}

The next step towards the proof of Theorem~\ref{th.sc1} is the time equicontinuity of the sequence $(f_n)_{n\ge 1}$ which we prove now.

\begin{lemma}\label{le.sc2}
There is $C_2>0$ depending only on $K$ and $f^{in}$ such that
\begin{equation}
\left\| f_n(t) - f_n(s) \right\|_1 \le C_2 (t-s)\,, \quad 0\le s\le t\,, \quad n\ge 1\,. \label{sc26}
\end{equation}
\end{lemma}

\begin{proof} Fix $R>0$. We infer from \eqref{sc1}, Fubini's theorem, \eqref{sc20}, and \eqref{sc21} that
\begin{align*}
\left\| \partial_t f_n(t) \right\|_1 & \le \frac{1}{2} \int_0^\infty \int_y^\infty K(y,x-y) f_n(t,x) f_n(t,y)\ dxdy \\
& \qquad + \int_0^\infty \int_0^\infty K(x,y) f_n(t,x) f_n(t,y)\ dydx \\
& \le \frac{3\kappa}{2} \int_0^\infty \int_0^\infty (1+x) (1+y) f_n(t,x) f_n(t,y)\ dydx \\
& \le \frac{3\kappa}{2} \|f_n(t)\|_{1,1}^2 \le C_2 := \frac{3\kappa}{2} \|f^{in}\|_{1,1}^2\,,
\end{align*}
from which \eqref{sc26} readily follows. 
\end{proof}

We are now in a position to complete the proof of Theorem~\ref{th.sc1}.

\begin{proof}[Proof of Theorem~\ref{th.sc1}] We first recall that, owing to the de la Vall\'ee Poussin theorem (in the form stated in Corollary~\ref{cob119}), the integrability of $f^{in}$ ensures that there is $\Phi\in\mathcal{C}_{VP,\infty}$ such that
\begin{equation}
\int_0^\infty \Phi(f^{in}(x))\ dx < \infty\,. \label{sc27}
\end{equation}
We then combine Lemma~\ref{le.sc1} (with $\Psi=\Phi$) and  \eqref{sc27} to conclude that, for each $t\ge 0$, $n\ge 1$, and $R>0$, 
\begin{equation}
\int_0^R \Phi(f_n(t,x))\ dx \le \|\Phi(f^{in})\|_1 e^{C_1(R) t}\ , \label{sc30}
\end{equation}
where $C_1(R)$ only depends on $K$, $f^{in}$, and $R$.

\smallskip

\noindent\textbf{Step~1: Weak compactness.} According to a variant of the Arzel\`a-Ascoli theorem (see \cite[Theorem~1.3.2]{Vr95} for instance), the sequence $(f_n)_{n\ge 1}$ is relatively sequentially compact in $C([0,T];w-L^1(0,\infty))$ for every $T>0$ if it enjoys the following two properties:
\begin{equation}
\text{ The sequence $(f_n(t))_{n\ge 1}$ is weakly compact in $L^1(0,\infty)$ for each $t \ge 0$,} \label{sc28}
\end{equation}
and:
\begin{equation}
\text{ The sequence $(f_n)_{n\ge 1}$ is weakly equicontinuous in $L^1(0,\infty)$ at every $t\ge 0$,} \label{sc29}
\end{equation}
see \cite[Definition~1.3.1]{Vr95}. Recall that the space $C([0,\infty);w-L^1(0,\infty))$ is the space of functions $h$ which are continuous in time with respect to the weak topology of $L^1(0,\infty)$, that is, 
$$
t \mapsto \int_0^\infty \vartheta(x) h(t,x)\ dx \in C([0,\infty)) \;\;\text{ for all }\;\; \vartheta\in L^\infty(0,\infty)\,.
$$

We first prove \eqref{sc28}. To this end, we recall that we have already established \eqref{sc30} and note that \eqref{sc20} entails that
\begin{equation}
\int_R^\infty f_n(t,x)\ dx \le \frac{1}{R} \int_R^\infty x f_n(t,x)\ dx \le \frac{M_1(f^{in})}{R}\ . \label{sc31}
\end{equation}
Since $\Phi\in C_{VP,\infty}$, the properties \eqref{sc30} and \eqref{sc31} imply that the sequence $(f_n(t))_{n\ge 1}$ is uniformly integrable in $L^1(0,\infty)$ for each $t\ge 0$ while \eqref{sc31} ensures that the condition \eqref{b105} of the Dunford-Pettis theorem (Theorem~\ref{thb103}) is satisfied. We are thus in a position to apply the Dunford-Pettis theorem to obtain \eqref{sc28}.

Let us now turn to \eqref{sc29} and notice that Lemma~\ref{le.sc2} entails that $(f_n)_{n\ge 1}$ is equicontinuous for the strong topology of $L^1(0,\infty)$ at every $t\ge 0$ and thus also weakly equicontinuous in $L^1(0,\infty)$ at every $t\ge 0$, which completes the proof of \eqref{sc29}.

We have thereby established that 
$$
\text{ the sequence }\ (f_n)_{n\ge 1}\ \text{ is relatively sequentially compact in }\ C([0,T];w-L^1(0,\infty))
$$ 
for every $T>0$ and a diagonal process ensures that there are a subsequence of $(f_n)_{n\ge 1}$ (not relabeled) and $f\in C([0,\infty);w-L^1(0,\infty))$ such that
\begin{equation}
f_n \longrightarrow f \;\;\text{ in }\;\; C([0,T];w-L^1(0,\infty)) \;\;\text{ for all }\;\; T>0\,. \label{sc32}
\end{equation}
Since $f_n$ is non-negative and satisfies \eqref{sc20} for each $n\ge 1$, we readily deduce from the convergence \eqref{sc32} that $f(t)$ is non-negative and satisfies \eqref{sc24}. 

\smallskip

\noindent\textbf{Step~2: Convergence.} We now check that the function $f$ constructed in the previous step solves \eqref{sc1}-\eqref{sc2} in an appropriate sense. To this end, let us first consider $t>0$ and a function $\vartheta\in L^\infty(0,\infty)$ with compact support included in $(0,R_0)$ for some $R_0>0$. By \eqref{sc11}, 
\begin{equation}
\int_0^\infty \vartheta(x) \left( f_n(t,x) - f^{in}(x) \right)\ dx = \frac{1}{2}  \left( I_{1,n}(t) + I_{2,n}(t) + I_{3,n}(t) \right)\,, \label{sc34}
\end{equation}
with
\begin{align*}
I_{1,n}(t) & := \int_0^t \int_0^{R_0} \int_0^{R_0} \tilde{\vartheta}(x,y) K_n(x,y) f_n(s,x) f_n(s,y)\ dydxds\,, \\
I_{2,n}(t) & := \int_0^t\int_0^{R_0} \int_{R_0}^\infty \tilde{\vartheta}(x,y) K_n(x,y) f_n(s,x) f_n(s,y)\ dydxds\,, \\
I_{3,n}(t) & := \int_0^t \int_{R_0}^\infty \int_0^\infty \tilde{\vartheta}(x,y) K(x,y) f_n(s,x) f_n(s,y)\ dydxds\,.
\end{align*}

Let us first identify the limit of $I_{1,n}(t)$. Since $(K_n)_{n\ge 1}$ is a bounded sequence of $L^\infty((0,R_0)\times (0,R_0))$ by \eqref{sc21} and converges a.e. towards $K$ in $(0,R_0)\times (0,R_0)$, we infer from Proposition~\ref{prd2} and the convergence \eqref{sc32} that
\begin{equation}
\lim_{n\to\infty} I_{1,n}(t) = \int_0^t \int_0^{R_0} \int_0^{R_0} \tilde{\vartheta}(x,y) K(x,y) f(s,x) f(s,y)\ dydxds\,. \label{sc35}
\end{equation}

Next, $\tilde{\vartheta}(x,y) = - \vartheta(x)$ for $(x,y)\in (0,R_0) \times (R_0,\infty)$ and
$$
I_{2,n}(t) = - \int_0^t\int_0^{R_0} \int_{R_0}^\infty \vartheta(x) K_n(x,y) f_n(s,x) f_n(s,y)\ dydxds\,.
$$
For $R>R_0$, we split $I_{2,n}(t)$ into two parts
\begin{equation}
I_{2,n}(t) = I_{21,n}(t,R) + I_{22,n}(t,R) \label{sc35a}
\end{equation} 
with
\begin{align*}
I_{21,n}(t,R) & := - \int_0^t \int_0^{R_0} \int_{R_0}^R \vartheta(x) K_n(x,y) f_n(s,x) f_n(s,y)\ dydxds\ , \\
I_{22,n}(t,R) & := - \int_0^t \int_0^{R_0} \int_R^\infty \vartheta(x) K_n(x,y) f_n(s,x) f_n(s,y)\ dydxds\ .
\end{align*}
On the one hand we argue as for $I_{1,n}(t)$ to conclude that
\begin{equation}
\lim_{n\to\infty}  I_{21,n}(t,R) = - \int_0^t \int_0^{R_0} \int_{R_0}^R \vartheta(x) K(x,y) f(s,x) f(s,y)\ dydxds\,. \label{sc36}
\end{equation}
On the other hand, using \eqref{sc20} and \eqref{sc22}, we find
\begin{align}
\left| I_{22,n}(t,R) \right| & \le \|\vartheta\|_\infty \int_0^t \int_0^{R_0} \int_R^\infty \omega_{R_0}(y) y f_n(s,x) f_n(s,y)\ dydxds \nonumber \\
& \le \|\vartheta\|_\infty \int_0^t M_0(f_n(s)) M_1(f_n(s))\ ds \sup_{y\in (R,\infty)}\{\omega_{R_0}(y)\} \nonumber \\
& \le t \|\vartheta\|_\infty M_0(f^{in}) M_1(f^{in})\ \sup_{y\in (R,\infty)}\{\omega_{R_0}(y)\} \mathop{\longrightarrow}_{R\to\infty} 0 \,. \label{sc37}
\end{align}
Similarly, owing to \eqref{sc22} and \eqref{sc24} (recall that \eqref{sc24} has been established at the end of Step~1), 
\begin{align}
& \int_0^t \int_0^{R_0} \int_R^\infty \vartheta(x) K(x,y) f(s,x) f(s,y)\ dydxds \nonumber \\ 
& \qquad \le t \|\vartheta\|_\infty M_0(f^{in}) M_1(f^{in})\ \sup_{y\in (R,\infty)}\{\omega_{R_0}(y)\} \mathop{\longrightarrow}_{R\to\infty} 0 \,. \label{sc38}
\end{align}
Combining \eqref{sc35}-\eqref{sc38} and letting first $n\to\infty$ and then $R\to\infty$, we end up with 
\begin{equation}
\lim_{n\to\infty} I_{2,n}(t) = \int_0^t \int_0^{R_0} \int_{R_0}^\infty \tilde{\vartheta}(x,y) K(x,y) f(s,x) f(s,y)\ dydxds\,. \label{sc39}
\end{equation}

Finally,  $\tilde{\vartheta}(x,y) = - \vartheta(y)$ for $(x,y)\in (R_0,\infty) \times (0,R_0)$ and $\tilde{\vartheta}(x,y)=0$ if $(x,y)\in  (R_0,\infty)\times (R_0,\infty)$ so that
$$
I_{3,n}(t) = - \int_0^t\int_{R_0}^\infty \int_0^{R_0} \vartheta(y) K_n(x,y) f_n(s,x) f_n(s,y)\ dydxds\,,
$$
and we argue as for $I_{2,n}(t)$ to obtain
\begin{equation}
\lim_{n\to\infty} I_{3,n}(t) = \int_0^t \int_{R_0}^\infty \int_{0}^\infty \tilde{\vartheta}(x,y) K(x,y) f(s,x) f(s,y)\ dydxds\,. \label{sc40}
\end{equation}

Using once more the convergence \eqref{sc32}, we may use \eqref{sc35}, \eqref{sc39}, and \eqref{sc40} to pass to the limit as $n\to\infty$ in \eqref{sc34} and conclude that $f$ satisfies \eqref{sc23} for all functions $\vartheta\in L^\infty(0,\infty)$ with compact support. Thanks to \eqref{sc21} and \eqref{sc24}, a classical density argument allows us to extend the validity of \eqref{sc23} to arbitrary functions $\vartheta\in L^\infty(0,\infty)$.

\smallskip

\noindent\textbf{Step~3: Strong continuity.} We now argue as in the proof of Lemma~\ref{le.sc2} to strengthen the time continuity of $f$. More precisely, let $t\ge 0$, $s\in [0,t]$, and $\vartheta\in L^\infty(0,\infty)$. We infer from \eqref{sc21}, \eqref{sc23}, and \eqref{sc24} that 
\begin{align*}
& \left| \int_0^\infty (f(t,x)-f(s,x)) \vartheta(x)\ dx \right| \\ 
& \qquad \le \frac{3\kappa}{2} \|\vartheta\|_\infty \int_s^t \int_0^\infty \int_0^\infty (1+x) (1+y) f(\sigma,x) f(\sigma,y)\ dydxd\sigma \\
& \qquad \le \frac{3\kappa}{2} \|\vartheta\|_\infty \int_s^t \|f(\sigma)\|_{1,0}^2\ d\sigma \\
& \qquad \le \frac{3\kappa}{2} \|\vartheta\|_\infty \|f^{in}\|_{1,0}^2\ |t-s|\,.
\end{align*}
Therefore,
\begin{align*}
\|f(t)-f(s)\|_1 & = \sup_{\vartheta\in L^\infty(0,\infty)}\left\{ \frac{1}{\|\vartheta\|_\infty} \left| \int_0^\infty (f(t,x)-f(s,x)) \vartheta(x)\ dx \right| \right\} \\ 
& \le \frac{3\kappa}{2} \|f^{in}\|_{1,0}^2\ |t-s|\,,
\end{align*}
which completes the proof of Theorem~\ref{th.sc1}.
\end{proof}

\subsubsection{Linearly growing kernels}
\label{sec:322}

We next turn to coagulation kernels growing at most linearly at infinity, that is, we assume that there is $\kappa_1>0$ such that
\begin{equation}
0 \le K(x,y) = K(y,x) \le \kappa_1 (2+x+y)\,, \quad (x,y)\in (0,\infty)\times (0,\infty)\,. \label{sc41}
\end{equation}
Observe that coagulation kernels satisfying \eqref{sc41} also satisfy \eqref{sc21} but need not satisfy \eqref{sc22}. 

For this class of coagulation kernels, we establish the existence of mass-conserving solutions to \eqref{sc1}-\eqref{sc2} \cite{BC90, DS96, LM02, St89, St91}.

\begin{theorem}\label{th.sc2}
Assume that the coagulation kernel $K$ satisfies \eqref{sc41} and consider an initial condition $f^{in}$ satisfying \eqref{sc18}. There is a non-negative function 
$$
f\in C([0,\infty);L^1(0,\infty)) \cap L^\infty(0,\infty;L_1^1(0,\infty))
$$ 
satisfying \eqref{sc23} and
\begin{equation}
M_0(f(t)) \le M_0(f^{in}) \;\;\mbox{ and }\;\; M_1(f(t)) = M_1(f^{in})\,, \quad t\ge 0\,. \label{sc42}
\end{equation}
\end{theorem}

The main difference between the outcomes of Theorem~\ref{th.sc1} and Theorem~\ref{th.sc2} is the conservation of mass $M_1(f(t))=M_1(f^{in})$ for all $t\ge 0$ for coagulation kernels satisfying \eqref{sc41}. 

Since the growth assumption \eqref{sc41} is more restrictive than \eqref{sc21}, it is clear that the proof of Theorem~\ref{th.sc2} has some common features with that of Theorem~\ref{th.sc1}. In particular, both Lemma~\ref{le.sc1} and Lemma~\ref{le.sc2} are valid in that case as well. The main difference lies actually in the control of the behaviour of $f_n(t,x)$ for large values of $x$ which is provided by the boundedness of $M_1(f_n(t))$ in \eqref{sc20}. This turns out to be not sufficient for coagulation kernels satisfying \eqref{sc41} and we first show that this assumption is particularly well-suited to control higher moments. 

\begin{lemma}\label{le.sc3}
Let $\psi\in\mathcal{C}_{VP}$ be such that $x\mapsto \psi(1+x) f^{in}(x)\in L^1(0,\infty)$. There is a positive constant $C_3>0$ depending only on $K$ and $f^{in}$ such that
\begin{equation}
\int_0^\infty \psi(1+x) f_n(t,x)\ dx \le \left( \int_0^\infty \psi(1+x) f^{in}(x)\ dx \right) e^{C_3 t}\,, \quad t\ge 0\,, \quad n\ge 1\,, \label{sc43}
\end{equation}
the function $f_n$ being still the solution to \eqref{sc1}-\eqref{sc2} with $K_n$ instead of $K$, the coagulation kernel $K_n$ being defined by \eqref{sc19}.
\end{lemma}

\begin{proof} Fix $A>0$ and define $\psi_A(x) := \min\{ \psi(1+x) , \psi(1+A) \} = \psi\left( 1 + \min\{x,A\} \right)$ for $x>0$. Then $\psi_A\in L^\infty(0,\infty)$ and it follows from \eqref{sc11} that
\begin{equation*}
\frac{d}{dt} \int_0^\infty \psi_A(x) f_n(t,x)\ dx = \frac{1}{2} \int_0^\infty \int_0^\infty \tilde{\psi}_A(x,y) K_n(x,y) f_n(t,x) f_n(t,y)\ dydx\,. 
\end{equation*}
Owing to the monotonicity and non-negativity of $\psi$, we note that
\begin{align*}
\tilde{\psi}_A(x,y) & = \psi(1+x+y) - \psi(1+x) - \psi(1+y) \\
& \le \psi(2+x+y) - \psi(1+x) - \psi(1+y) \;\;\text{ if }\;\; y\in (0,A-x) \;\;\text{ and }\;\; x\in (0,A)\,, \\
\tilde{\psi}_A(x,y) & = \psi(1+A) - \psi(1+x) - \psi(1+y) \\
& \le \psi(2+x+y) - \psi(1+x) - \psi(1+y) \;\;\text{ if }\;\; y\in (A-x,A) \;\;\text{ and }\;\; x\in (0,A)\,, \\
\tilde{\psi}_A(x,y) & = - \psi(1+x) \le 0 \;\;\text{ if }\;\; (x,y)\in (0,A)\times (A,\infty)\,, \\
\tilde{\psi}_A(x,y) & = - \psi(1+y) \le 0 \;\;\text{ if }\;\; (x,y)\in (A,\infty \times (0,A)\,, \\
\tilde{\psi}_A(x,y) & = - \psi(1+A) \le 0 \;\;\text{ if }\;\; (x,y)\in (A,\infty \times (A,\infty)\,.
\end{align*}
Consequently, 
\begin{align*}
& \frac{d}{dt} \int_0^\infty \psi_A(x) f_n(t,x)\ dx \\
& \qquad \le \frac{1}{2} \int_0^A \int_0^A \left( \psi(2+x+y) - \psi(1+x) - \psi(1+y) \right) K_n(x,y) f_n(t,x) f_n(t,y)\ dydx\,,
\end{align*}
and we infer from \eqref{b125} and \eqref{sc41} that
\begin{align*}
& \left( \psi(2+x+y) - \psi(1+x) - \psi(1+y) \right) K_n(x,y) \\ 
& \qquad \le \kappa_1 \left( \psi(2+x+y) - \psi(1+x) - \psi(1+y) \right) (2+x+y) \\
& \qquad \le 2\kappa_1 \left( (1+x) \psi(1+y) + (1+y) \psi(1+x) \right)\,.
\end{align*}
Combining the above two estimates leads us to
$$
\frac{d}{dt} \int_0^\infty \psi_A(x) f_n(t,x)\ dx \le 2\kappa_1 \|f_n(t)\|_{1,1}\ \int_0^A \psi(1+x) f_n(t,x)\ dx\,.
$$
Using \eqref{sc20}, we end up with 
$$
\frac{d}{dt} \int_0^\infty \psi_A(x) f_n(t,x)\ dx \le 2\kappa_1 \|f^{in}\|_{1,1}\ \int_0^\infty \psi_A(x) f_n(t,x)\ dx\,,
$$
and \eqref{sc43} follows by integration with $C_3 := 2\kappa_1 \|f^{in}\|_{1,1}$. \end{proof}

\begin{proof}[Proof of Theorem~\ref{th.sc2}] As in the proof of Theorem~\ref{th.sc1}, the de la Vall\'ee Poussin theorem (Theorem~\ref{thb106}) ensures the existence of a function $\Phi\in \mathcal{C}_{VP,\infty}$ such that $\Phi(f^{in})\in L^1(0,\infty)$. Moreover, observing that the property $f^{in}\in L^1(0,\infty;(1+x)\ dx)$ also reads $x\mapsto 1+x \in L^1(0,\infty;f^{in}(x)\ dx)$, we use once more the de la Vall\'ee Poussin theorem (now with $\mu=f^{in} dx$) to obtain a function $\varphi\in\mathcal{C}_{VP,\infty}$ such that $x\mapsto \varphi(1+x)$ belongs to $L^1(0,\infty;f^{in}(x)\ dx)$. Summarizing we have established that there are two functions $\Phi$ and $\varphi$ in $\mathcal{C}_{VP,\infty}$ such that
\begin{equation}
\int_0^\infty \left[ \Phi(f^{in}(x)) + \varphi(1+x) f^{in}(x) \right]\ dx < \infty\ . \label{sc44}
\end{equation}
We now infer from \eqref{sc44}, Lemma~\ref{le.sc1} (with $\Psi=\Phi$) and Lemma~\ref{le.sc3} (with $\psi=\varphi$) that, for each $t\ge 0$, $n\ge 1$, and $R>0$, 
\begin{align}
\int_0^R \Phi(f_n(t,x))\ dx & \le \|\Phi(f^{in})\|_1 e^{C_1(R) t}\ , \label{sc44a} \\
\int_0^\infty \varphi(1+x) f_n(t,x)\ dx & \le \left( \int_0^\infty \varphi(1+x) f^{in}(x)\ dx \right) e^{C_3t}\ , \label{sc44b}
\end{align}
where $C_1(R)$ only depends on $K$, $f^{in}$, and $R$ and $C_3$ on $K$ and $f^{in}$.

\smallskip

\noindent\textbf{Step~1: Weak compactness.} We argue as in the first step of the proof of Theorem~\ref{th.sc1} with the help of \eqref{sc20} and \eqref{sc44a} to conclude that there are a subsequence of $(f_n)_{n\ge 1}$ (not relabeled) and a non-negative function $f\in C([0,\infty);w-L^1(0,\infty))$ such that
\begin{equation}
f_n \longrightarrow f \;\;\text{ in }\;\; C([0,T];w-L^1(0,\infty)) \;\;\text{ for all }\;\; T>0\,. \label{sc45}
\end{equation}

\smallskip

\noindent\textbf{Step~2: Convergence.} We keep the notation used in the proof of Theorem~\ref{th.sc1} and notice that \eqref{sc35} and \eqref{sc36} are still valid. A different treatment is required for $I_{22,n}(t,R)$: thanks to \eqref{sc20}, \eqref{sc41}, \eqref{sc44}, \eqref{sc44b}, and the monotonicity of $r\mapsto \varphi(r)/r$, we obtain
\begin{align*}
I_{22,n}(t,R) & \le \kappa \|\vartheta\|_\infty \int_0^t \int_0^{R_0} \int_R^\infty (2+x+y) f_n(s,x) f_n(s,y)\ dydxds \\
& \le 2\kappa \|\vartheta\|_\infty \int_0^t \int_0^{R_0} \int_R^\infty (1+x)(1+y) f_n(s,x) f_n(s,y)\ dydxds \\
& \le \frac{2 \kappa R}{\varphi(R)} \|\vartheta\|_\infty \int_0^t \|f_n(s)\|_{1,1} \int_R^\infty \varphi(1+y) f_n(s,y)\ dyds \\
& \le \frac{2\kappa R}{\varphi(R)} \|\vartheta\|_\infty \|f^{in} \|_{1,1} \left( \int_0^\infty \varphi(1+y) f^{in}(y)\ dy \right) \frac{e^{C_3 t}}{C_3}\,,
\end{align*}
and the right-hand side of the above inequality converges to zero as $R\to\infty$ since $\varphi\in \mathcal{C}_{VP,\infty}$. Arguing in a similar way to handle $I_{3,n}(t)$, we complete the proof of \eqref{sc23} as in the proof of Theorem~\ref{th.sc1}. Finally, the mass conservation $M_1(f(t))=M_1(f^{in})$ for each $t>0$ follows by passing to the limit as $n\to\infty$ in the equality $M_1(f_n(t))=M_1(f^{in})$ from \eqref{sc20} with the help of \eqref{sc44}, \eqref{sc44b}, \eqref{sc45}, and the property $\varphi\in \mathcal{C}_{VP,\infty}$ to control the behaviour for large values of $x$. 

\end{proof}

\subsubsection{Product kernels}
\label{sec:323}

The last class of kernels we consider allows us to get rid of any growth condition on $K$ provided it has a specific form. More precisely, we assume that there is a non-negative continuous function $r\in C([0,\infty))$ such that $r(x)>0$ for $x>0$ and 
\begin{equation}
K(x,y) = r(x) r(y)\ , \quad (x,y)\in (0,\infty)\times (0,\infty)\ . \label{sc46}
\end{equation}
The celebrated multiplicative kernel $K_2(x,y) =xy$ fits into this framework with $r(x)=x$ for $x>0$. Observe that no growth condition is required on $r$. 

For this class of kernels, the existence result is similar to Theorem~\ref{th.sc1} and reads:

\begin{theorem}\label{th.sc3}
Assume that the coagulation kernel $K$ satisfies \eqref{sc46} and consider an initial condition $f^{in}$ satisfying \eqref{sc18}. There is a non-negative function 
$$
f\in C([0,\infty);L^1(0,\infty)) \cap L^\infty(0,\infty;L_1^1(0,\infty))
$$ 
satisfying \eqref{sc23} and
\begin{equation}
M_0(f(t)) \le M_0(f^{in}) \;\;\mbox{ and }\;\; M_1(f(t)) \le M_1(f^{in})\,, \quad t\ge 0\,. \label{sc47}
\end{equation}
\end{theorem}

To prove Theorem~\ref{th.sc3}, it turns out that it is easier to work with the following truncated version of $K$ which differs from \eqref{sc19}. Given an integer $n\ge 1$ and $x>0$, we define $r_n(x) := \min\{ r(x) , n \}$ and
\begin{equation}
\tilde{K}_n(x,y) := r_n(x) r_n(y) \ , \quad (x,y)\in (0,\infty)\times (0,\infty)\ . \label{sc48}
\end{equation}
Clearly, $\tilde{K}_n$ is a non-negative, bounded, and symmetric function and we infer from \eqref{sc18} and Proposition~\ref{pr.sc1} that the initial-value problem \eqref{sc1}-\eqref{sc2} with $\tilde{K}_n$ instead of $K$ has a unique non-negative solution $g_n\in C^1([0,\infty);L^1(0,\infty))$ which satisfies
\begin{equation}
M_0(g_n(t)) \le M_0(f^{in}) \;\;\mbox{ and }\;\; M_1(g_n(t)) = M_1(f^{in})\ , \quad t\ge 0\ . \label{sc49}
\end{equation}
Without a control on the growth of $K$, the estimate \eqref{sc49} on $\|g_n(t)\|_{1,1}$ is not sufficient to control the behaviour of $g_n(t,x)$ for large values of $x$. As we shall see now, it is the specific form \eqref{sc46} of $K$ which provides this control.

\begin{lemma}\label{le.sc4}
For $t>0$, $A>0$, and $n\ge 1$,
\begin{align}
\int_0^t \left( \int_0^\infty r_n(x) g_n(s,x)\ dx \right)^2\ ds & \le 2M_0(f^{in})\ , \label{sc50} \\
\int_0^t \left( \int_A^\infty r_n(x) g_n(s,x)\ dx \right)^2\ ds & \le \frac{2M_1(f^{in})}{A}\ . \label{sc51}
\end{align}
\end{lemma}

\begin{proof} On the one hand, the bound \eqref{sc50} readily follows from \eqref{sc11} (with $\vartheta\equiv 1$), \eqref{sc48}, and the non-negativity of $g_n$. On the other hand, let $A>0$ and define $\vartheta_A(x) := \min\{x,A\}$ for $x>0$ as in the proof of Proposition~\ref{pr.sc1}. Owing to \eqref{sc13}, we deduce from \eqref{sc11} and the non-negativity of $g_n$ that 
\begin{align*}
\int_0^t \int_A^\infty \int_A^\infty \tilde{K}_n(x,y) g_n(s,x) g_n(s,y)\ dydxds & \le \frac{2}{A} \int_0^\infty \vartheta_A(x) f^{in}(x)\ dx \le \frac{2 M_1(f^{in})}{A}\ .
\end{align*}
Combining \eqref{sc48} and the above inequality gives \eqref{sc51}. 
\end{proof}

We next derive the counterpart of the equicontinuity property established in Lemma~\ref{le.sc2}.

\begin{lemma}\label{le.sc5} There is a modulus of continuity $\omega$ (that is, a function $\omega: (0,\infty)\to [0,\infty)$ satisfying $\omega(z)\to 0$ as $z\to 0$) such that
\begin{equation}
\| g_n(t) - g_n(s)\|_1 \le \omega(t-s)\ , \quad 0 \le s \le t\ , \quad n\ge 1\ . \label{sc52}
\end{equation}
\end{lemma}

\begin{proof} For $R>0$ we introduce 
$$
m(R) := \sup_{x\in (0,R)}\left\{ \frac{r(x)}{1+x} \right\}\ , 
$$
which is well-defined according to the continuity of $r$. We infer from \eqref{sc1} and Fubini's theorem that
\begin{align*}
\int_0^R |g_n(t,x)-g_n(s,x)|\ dx & \le \int_s^t \int_0^R \int_y^R \tilde{K}_n(y,x-y) g_n(\sigma,x-y) g_n(\sigma,y)\ dxdyd\sigma \\
& \qquad + \int_s^t \int_0^R \int_0^\infty \tilde{K}_n(x,y) g_n(\sigma,x) g_n(\sigma,y)\ dydxd\sigma \\
& = \int_s^t \int_0^R \int_0^{R-y} r_n(x) r_n(y) g_n(\sigma,x) g_n(\sigma,y)\ dxdyd\sigma \\
& \qquad + \int_s^t \int_0^R \int_0^\infty r_n(x) r_n(y) g_n(\sigma,x) g_n(\sigma,y)\ dydxd\sigma \\
& \le 2 \int_s^t \int_0^R \int_0^\infty r_n(x) r_n(y) g_n(\sigma,x) g_n(\sigma,y)\ dydxd\sigma\ .
\end{align*}
We next use \eqref{sc49}, H\"older's inequality, and \eqref{sc50} to obtain
\begin{align*}
& \int_0^R |g_n(t,x)-g_n(s,x)|\ dx \\ 
& \qquad \le 2 \int_s^t m(R) \|g_n(\sigma)\|_{1,1} \int_0^\infty r_n(y) g_n(\sigma,y)\ dyd\sigma \\
& \qquad \le 2 m(R) \|f^{in}\|_{1,1} \sqrt{t-s} \left[ \int_s^t \left( \int_0^\infty r_n(y) g_n(\sigma,y)\ dy \right)^2 d\sigma \right]^{1/2} \\
& \qquad \le \left( 2 \|f^{in}\|_{1,1} \right)^{3/2} m(R) \sqrt{t-s}\ .
\end{align*}
Combining \eqref{sc49} and the above inequality leads us to
\begin{align*}
\| g_n(t) - g_n(s)\|_1 & \le \int_0^R |g_n(t,x)-g_n(s,x)|\ dx + \int_R^\infty (g_n(t,x)+g_n(s,x))\ dx \\
& \le \left( 2 \|f^{in}\|_{1,1} \right)^{3/2} m(R) \sqrt{t-s} + \frac{2}{R} M_1(f^{in})\ ,
\end{align*}
hence
\begin{equation}
\| g_n(t) - g_n(s)\|_1 \le C_4 \frac{1+Rm(R) \sqrt{t-s}}{R} \label{sc53}
\end{equation}
for some positive constant $C_4$ depending only on $f^{in}$.

Now, $R\mapsto R m(R)$ is a non-decreasing and continuous function such that $R m(R)\to 0$ as $R\to 0$, $Rm(R)>0$ for $R>0$, and $R m(R) \to\infty$ as $R\to \infty$. Introducing its generalized inverse 
$$
Q(z) := \inf\left\{ R\ge 0\ :\ Rm(R) \ge z \right\}\ , \quad z\ge 0\ ,
$$
the function $Q$ is also non-decreasing with 
$$
Q(0)=0\ , \quad Q(z)>0 \;\;\text{ for }\;\; z>0\ , \;\;\text{ and }\;\; Q(z)\to\infty \;\;\text{ as }\;\; z\to\infty\ . 
$$
Setting
$$
\frac{1}{\omega(z)} := \frac{3}{4} Q\left( \frac{1}{z} \right)\ , \quad z>0\ ,
$$
and choosing $R=1/\omega(\sqrt{t-s})$ in \eqref{sc53} we end up with
$$
\| g_n(t) - g_n(s)\|_1 \le 2C_4 \omega(\sqrt{t-s})\ .
$$
The properties of $Q$ ensure that $2C_4 \omega$ is a modulus of continuity and the proof of Lemma~\ref{le.sc5} is complete. 
\end{proof}

\begin{proof}[Proof of Theorem~\ref{th.sc3}] The proof proceeds along the same lines as that of Theorem~\ref{th.sc1}, the control on the behaviour for large $x$ and the time equicontinuity being provided by Lemma~\ref{le.sc4} and Lemma~\ref{le.sc5} instead of \eqref{sc20} and Lemma~\ref{le.sc2}. Note that Lemma~\ref{le.sc1} is still valid owing to the local boundedness of $r$. 
\end{proof}

\begin{remark}\label{re.sc2}
As in \cite{Lt00}, it is possible to extend Theorem~\ref{th.sc3} to perturbations of product kernels of the form $K(x,y) = r(x) r(y) + \tilde{K}(x,y)$ provided $0\le \tilde{K}(x,y) \le \kappa_1 r(x) r(y)$ for $x>0$, $y>0$, and some $\kappa_1>0$.
\end{remark}

Another peculiar extension of Theorem~\ref{th.sc3} is the possibility of constructing mass-conserving solutions for coagulation kernels of the form \eqref{sc46} satisfying $r(x)/\sqrt{x}\to\infty$ as $x\to\infty$. 

\begin{proposition}\label{pr.sc6}
Assume that the coagulation kernel $K$ satisfies \eqref{sc46} and that $r\in C([0,\infty))\cap C^1((0,\infty))$ is a concave and positive function such that
\begin{equation}
\int_1^\infty \frac{dx}{r(x)^2} = \infty\ . \label{sc100}
\end{equation}
Consider an initial condition $f^{in}$ satisfying \eqref{sc18} together with $f^{in}\in L_2^1(0,\infty)$. Then there exists a solution $f$ to \eqref{sc1}-\eqref{sc2} such that $M_1(f(t))=M_1(f^{in})$ for each $t\ge 0$.
\end{proposition}

The function $r(x) = \sqrt{2+x} \left( \ln{(2+x)} \right)^{\alpha}$ satisfies the assumptions of Proposition~\ref{pr.sc6} for $\alpha\in (0,1/2]$.

\begin{proof}[Proof of Proposition~\ref{pr.sc6}] We keep the notations of the proof of Theorem~\ref{th.sc3}, the existence part of Proposition~\ref{pr.sc6} being a consequence of it. To show that the solution $f$ constructed in the proof of Theorem~\ref{th.sc3} is mass-conserving, we check that \eqref{sc100} allows us to control the time evolution of the second moment $M_2(f(t))$ for all times $t\ge 0$. Indeed, for $n\ge 1$, we deduce from \eqref{sc11} that 
$$
\frac{d}{dt}M_2(g_n(t)) = \left( \int_0^\infty r_n(x)\ x\ g_n(t,x)\ dx \right)^2\ .
$$
Since $r$ is concave, so is $r_n$ and Jensen's inequality and \eqref{sc49} ensure that
$$
\frac{d}{dt}M_2(g_n(t)) \le M_1(f^{in})^2 \ \left[ r\left( \frac{M_2(g_n(t))}{M_1(f^{in})} \right) \right]^2\ .
$$
Now, the assumption \eqref{sc100} guarantees that the ordinary differential equation
$$
\frac{dY}{dt} = M_1(f^{in})^2 \ \left[ r\left( \frac{Y}{M_1(f^{in})}
\right) \right]^2\ , \quad t\ge 0\ ,
$$
has a global solution $Y\in C^1([0,\infty))$ satisfying $Y(0)=M_2(f^{in})$ which is locally bounded. The comparison principle then implies that $M_2(g_n(t))\le Y(t)$ for all $t\ge 0$ and $n\ge 1$. We next use the convergence of $(g_n)_{n\ge 1}$ towards $f$ to conclude that $M_2(f(t))\le Y(t)$ for all $t\ge 0$. We finally combine this information with \eqref{sc49} to show that $M_1(f(t)) = M_1(f^{in})$ for $t>0$ and complete the proof. 
\end{proof}

\subsection{Gelation}
\label{sec:33}

In Section~\ref{sec:322} we have shown the existence of mass-conserving solutions to \eqref{sc1}-\eqref{sc2} for coagulation kernels satisfying the growth condition \eqref{sc41}. As already mentioned this property fails to be true in general for coagulation kernels which grows sufficiently fast for large $x$ and $y$, a fact which has been known/conjectured since the early eighties \cite{EZH84, HEZ83, LT82, Zi80} but only proved recently in \cite{EMP02, Je98}. In fact, the occurrence of gelation was first shown for the multiplicative kernel $K_2(x,y)=xy$ by an elementary argument \cite{LT81} and conjectured to take place for coagulation kernels $K$ satisfying $K(x,y)\ge \kappa_m (xy)^{\lambda/2}$ for some $\lambda\in (1,2]$ and $\kappa_m>0$ \cite{EZH84, HEZ83, LT82, Zi80}. This conjecture was supported by a few explicit solutions constructed in \cite{dC98, vDE85, Le83}. A first breakthrough was made in \cite{Je98} where a stochastic approach is used to show that, for a dense set of initial data, there exists at least one gelling solution to the discrete coagulation equations. A definitive and positive answer is provided in \cite{EMP02} where the occurrence of gelation in finite time is proved for all weak solutions starting from an arbitrary initial condition $f^{in}$ satisfying \eqref{sc18} and $f^{in}\not\equiv 0$. 

More precisely, let $K$ be a non-negative and symmetric function such that 
\begin{equation}
K(x,y) \ge r(x) r(y)\ , \quad (x,y)\in (0,\infty)\times (0,\infty)\ , \label{sc54}
\end{equation}
for some non-negative function $r$. Consider an initial condition $f^{in}$, $f^{in}\not\equiv 0$, satisfying \eqref{sc18} and let
$$
f\in C([0,\infty);L^1(0,\infty)) \cap L^\infty(0,\infty;L_1^1(0,\infty))\ , \quad f\ge 0\ ,
$$
be a solution to \eqref{sc1}-\eqref{sc2} satisfying \eqref{sc23} and \eqref{sc24} (such a solution exists for a large class of coagulation kernels, see Theorem~\ref{th.sc1} and Theorem~\ref{th.sc3}).

\begin{theorem}[\cite{EMP02}]\label{th.sc4}
If there are $\lambda\in (1,2]$ and $\kappa_m>0$ such that
$r(x)=\kappa_m x^{\lambda/2}$, $x>0$, then gelation occurs in finite time, that is, there is $T_{gel}\in [0,\infty)$ such that
$$
M_1(f(t)) < M_1(f^{in}) \;\;\text{ for }\;\; t>T_{gel}\ .
$$
\end{theorem}

The cornerstone of the proof of Theorem~\ref{th.sc4} is the following estimate. 

\begin{proposition}[\cite{EMP02}]\label{pr.sc2}
Let $\xi~:[0,\infty)\longrightarrow [0,\infty)$ be a non-decreasing differentiable function satisfying $\xi(0)=0$ and 
\begin{equation}
I_\xi := \int_0^\infty \xi'(A)\ A^{-1/2}\ dA < \infty\ . \label{sc55}
\end{equation}
Then, for $t>0$, 
\begin{equation}
\int_{0}^{t} \left( \int_0^\infty r(x) \xi(x)\ f(s,x)\ dx \right)^2\ ds \le 2 I_\xi^2\ M_1(f^{in})\ . \label{sc56}
\end{equation}
\end{proposition}

Let us mention at this point that, besides paving the way to a proof of the occurrence of gelation in finite time, other important consequences can be drawn from Proposition~\ref{pr.sc2} due to the possibility of choosing different functions $\xi$. These consequences include temporal decay estimates for large times as well as more precise information on $f$ across the gelation time \cite{EMP02}. 

Before proving Proposition~\ref{pr.sc2}, let us sketch how to use it to establish Theorem~\ref{th.sc4}. Since $r(x) = \kappa_m x^{\lambda/2}$, a close look at \eqref{sc56} indicates that the choice $\xi(x) = x^{(2-\lambda)/2}$ gives
$$
\int_0^t M_1(f(s))^2\ ds = \int_{0}^{t} \left( \int_0^\infty x^{\lambda/2}\ \xi(x)\ f(s,x)\ dx \right)^2\ ds \le \frac{2 I_\xi^2\ M_1(f^{in})}{\kappa_m^2}
$$
for all $t>0$. Consequently, $t\mapsto M_1(f(t))$ belongs to $L^2(0,\infty)$ and thus $M_1(f(t))$ cannot remain constant throughout time evolution. However $\xi$ does not satisfy \eqref{sc55} as
$$
\int_1^\infty \xi'(A)\ A^{-1/2}\ dA < \infty \;\;\text{ but }\;\; \int_0^1 \xi'(A)\ A^{-1/2}\ dA = \infty\ .
$$
We shall see below that a suitable choice is $\xi(x) = (x-1)_+^{(2-\lambda)/2}$. 

\begin{proof}[Proof of Proposition~\ref{pr.sc2}]
Fix $A>0$ and take $\vartheta_A(x)=\min{\{x,A\}}$, $x>0$, in \eqref{sc23}. Recalling \eqref{sc13}, we deduce from the non-negativity of $f$ that
$$
\int_0^t \int_A^\infty \int_A^\infty K(x,y) f(s,x) f(s,y)\ dydxds \le \frac{2}{A} \int_0^\infty \vartheta_A(x) f^{in}(x)\ dx \le \frac{2 M_1(f^{in})}{A}\ .
$$
Using \eqref{sc54} we end up with
\begin{equation}
\int_0^t \left( \int_A^\infty r(x) f(s,x)\ dx \right)^2\ ds \le \frac{2M_1(f^{in})}{A}\ . \label{sc57}
\end{equation}
We then infer from \eqref{sc57}, Fubini's theorem, and Cauchy-Schwarz' inequality that
\begin{align*}
& \int_{0}^{t} \left( \int_0^\infty r(x) \xi(x)\ f(s,x)\ dx \right)^2\ ds \\
& \qquad = \int_{0}^{t} \left( \int_0^\infty \int_0^x r(x) \xi'(A)\ f(s,x)\ dAdx \right)^2\ ds \\
& \qquad = \int_{0}^{t} \left( \int_0^\infty \xi'(A) \int_A^\infty r(x) f(s,x)\ dxdA \right)^2\ ds \\
& \qquad \le I_\xi \int_0^t \int_0^\infty \xi'(A) \sqrt{A} \left( \int_A^\infty r(x) f(s,x)\ dx \right)^2 dAds \\
& \qquad \le 2 M_1(f^{in}) I_\xi \int_0^\infty \frac{\xi'(A)}{\sqrt{A}}\ dA\ ,
\end{align*}
hence \eqref{sc56}. 
\end{proof}

\begin{proof}[Proof of Theorem~\ref{th.sc4}] It first follows from \eqref{sc23} with $\vartheta\equiv 1$, \eqref{sc54}, the choice of $r$, and the non-negativity of $f$ that
$$
\int_0^t \left( \int_0^\infty x^{\lambda/2} f(s,x)\ dx \right)^2\ ds \le \frac{2 M_0(f^{in})}{\kappa_m^2}\ ,
$$
which implies, since $x \le 2^{(2-\lambda)/2} x^{\lambda/2}$ for $x\in (0,2)$,
\begin{equation}
\int_0^t \left( \int_0^2 x f(s,x)\ dx \right)^2\ ds \le \frac{2^{3-\lambda} M_0(f^{in})}{\kappa_m^2}\ . \label{sc58}
\end{equation}
We next take $\xi(A) = (A-1)_+^{(2-\lambda)/2}$, $A>0$, in Proposition~\ref{pr.sc2} and note that
$$
I_\xi = \frac{2-\lambda}{2} \int_1^\infty (A-1)^{-\lambda/2} A^{-1/2} \ dA < \infty
$$
as $\lambda>1$. Since $x-1\ge x/2$ for $x\ge 2$, we infer from \eqref{sc56} that 
\begin{align}
\frac{1}{2^{2-\lambda}} \int_{0}^{t} \left( \int_2^\infty x f(s,x)\ dx \right)^2\ ds & \le \int_{0}^{t} \left( \int_0^\infty x^{\lambda/2}\ (x-1)_+^{(2-\lambda)/2}\ f(s,x)\ dx \right)^2\ ds \nonumber \\
& \le \frac{2 I_\xi^2\ M_1(f^{in})}{\kappa_m^2}\ . \label{sc59}
\end{align}
Combining \eqref{sc58} and \eqref{sc59} implies that $t\mapsto M_1(f(t))$ belongs to $L^2(0,\infty)$ and thus $M_1(f(t))$ cannot remain constant throughout time evolution. Since $M_1(f(t))\le M_1(f^{in})$ for all $t>0$ by \eqref{sc24}, we conclude that $T_{gel}<\infty$. 
\end{proof}

Using the same approach, we can actually extend Theorem~\ref{th.sc4} to a slightly wider setting encompassing the power functions.

\begin{proposition}\label{pr.sc3}
Assume that $r\in C([0,\infty))\cap C^1((0,\infty))$ is a concave and positive function which satisfies also
\begin{equation}
\int_1^\infty \frac{dx}{x^{1/2}\ r(x)} < \infty \;\;\text{ and }\;\;
\lim_{x\to \infty}\ \frac{r(x)}{x^{1/2}} = \lim_{x\to \infty}\
\frac{x}{r(x)} = \infty\ , \label{sc60} 
\end{equation}
as well as $r(x)\ge \delta x$ for $x\in (0,1)$ for some $\delta>0$. Then $T_{gel} < \infty$. 
\end{proposition}

A typical example of function $r$ satisfying all the assumptions of Proposition~\ref{pr.sc3} is a positive and concave function behaving as $\sqrt{x} \left( \ln{x} \right)^{1+\alpha}$ for large $x$ for some $\alpha>0$.

\begin{proof} The proof is similar to that of Theorem~\ref{th.sc4}, the main difference being the choice of the function $\xi$ in the use of Proposition~\ref{pr.sc2}. As $r$ is concave and positive, the function $x\mapsto x/r(x)$ is non-decreasing and we set
$$
\xi(x) := \left( \frac{x}{r(x)} - \frac{1}{r(1)} \right)_+\ , \quad
x > 0\ .
$$ 
Then $\xi$ is a positive and non-decreasing differentiable function with $\xi(0)=0$. Moreover, \eqref{sc60} guarantees that
\begin{align*}
I_\xi & = \left[ \frac{A}{r(A)} A^{-1/2} \right]_{A=1}^{A=\infty} + \frac{1}{2}  \int_1^\infty \frac{A}{r(A)} A^{-3/2}\ dA \\
& = - \frac{1}{r(1)} + \frac{1}{2}\ \int_1^\infty \frac{dA}{r(A)\ A^{1/2}} < \infty\ .
\end{align*}
We are therefore in a position to apply Proposition~\ref{pr.sc2}. Since there is $x_\star>1 $ such that $x/r(x)\ge 2/r(1)$ for $x\ge x_\star$ by \eqref{sc60}, we deduce from \eqref{sc56} that 
\begin{equation}
\int_0^t \left( \int_{x_\star}^\infty x\ f(s,x)\ dx
\right)^2\ ds \le 8 I_\xi^2 M_1(f^{in})\ . \label{sc61}
\end{equation}
We finally infer from \eqref{sc23} with $\vartheta\equiv 1$, \eqref{sc54}, and the non-negativity of $f$ that 
$$
\int_0^t \left( \int_0^\infty r(x)\ f(s,x)\ dx
\right)^2\ ds \le 2 M_0(f^{in}) \ ,
$$
and the assumptions on $r$ ensure that there is $\delta_\star\in (0,\delta)$ such that $\delta_\star x\le r(x)$ for $x\in
(0,x_\star)$. Consequently,
$$
\int_0^t \left( \int_0^{x_\star} x f(s,x)\ dx \right)^2\ ds \le \frac{2 M_0(f^{in})}{\delta_\star^2} \ .
$$
Combining this estimate with \eqref{sc61} allows us to conclude that $t\mapsto M_1(f(t))\in L^2(0,\infty)$ and complete the proof. 
\end{proof}

\subsection{Uniqueness}
\label{sec:34}

The uniqueness issue has been investigated by several authors but the results obtained so far are restricted to mass-conserving solutions, an exception being the multiplicative kernel $K_2(x,y)=xy$. Actually two approaches have been developed to establish the uniqueness of solutions to \eqref{sc1}-\eqref{sc2}: a direct one which consists in taking two solutions and estimating a weighted $L^1$-norm of their difference and another one based on a kind of Wasserstein distance. To be more specific, since the pioneering works \cite{ML64, Me57}, uniqueness has been proved in \cite{BC90, DS96, Gi13, Jo03, Lt02, No99, St90} by the former approach and is summarized in the next result.

\begin{proposition}\label{pr.sc4}
Assume that there is a non-negative subadditive function $\varphi$ (that is, $\varphi(x+y) \le \varphi(x)+\varphi(y)$, $x>0$, $y>0$), such that 
$$
K(x,y) \le \varphi(x)\varphi(y)\ , \quad x>0\ , \ y>0\ .
$$
Let $T>0$ and $f^{in}\in L^1(0,\infty;\varphi(x)\ dx)$, $f^{in}\ge 0$. There is at most one solution 
$$
f\in C([0,T];L^1(0,\infty;\varphi(x) dx)) \cap L^1(0,T; L^1(0,\infty;\varphi(x)^2 dx))
$$
to \eqref{sc1}-\eqref{sc2}. 
\end{proposition}

Let us point out two immediate consequences of Proposition~\ref{pr.sc4}. First, if $\varphi(x)\le x^{1/2}$, Proposition~\ref{pr.sc4} implies the uniqueness of the solution constructed in Theorem~\ref{th.sc2} since it belongs to $L^\infty(0,\infty;L_1^1(0,\infty))$. Next, if $\varphi(x)=x$, it gives the uniqueness of solutions to \eqref{sc1}-\eqref{sc2} as long as $M_2(f)\in L^1(0,T)$. While this is only true up to a finite time $T$ in the general framework considered in Theorem~\ref{th.sc1}, it follows from Theorem~\ref{th.sc2} and Lemma~\ref{le.sc3} (with $\psi(x) = x^2$) that, if $f^{in}$ belongs to $L_2^1(0,\infty)$ and satisfies \eqref{sc18}, then the solution $f$ to \eqref{sc1}-\eqref{sc2} constructed in Theorem~\ref{th.sc2} is such that $t\mapsto M_2(f(t)) \in L^\infty(0,T)$ for any $T>0$. According to Proposition~\ref{pr.sc4}, this solution is unique.

\begin{proof}[Proof of Proposition~\ref{pr.sc4}] Let $f_1$ and $f_2$ be two solutions to \eqref{sc1}-\eqref{sc2} enjoying the properties listed in Proposition~\ref{pr.sc4}. We infer from \eqref{sc1} that
\begin{align*}
& \frac{d}{dt} \int_0^\infty |(f_1-f_2)(t,x)| \varphi(x)\ dx \\
& \qquad = \frac{1}{2} \int_0^\infty \int_0^\infty K(x,y) (f_1+f_2)(t,y) (f_1-f_2)(t,x) \Theta(t,x,y)\ dydx
\end{align*}
with $\sigma := \text{ sign}(f_1-f_2)$ and
$$
\Theta(t,x,y) := \left[ (\varphi\sigma)(t,x+y) - (\varphi\sigma)(t,x) -(\varphi\sigma)(t,y) \right]\ .
$$
Observing that
\begin{align*}
& (f_1-f_2)(t,x) \Theta(t,x,y) = |(f_1-f_2)(t,x)| \sigma(t,x) \Theta(t,x,y) \\
& \qquad = |(f_1-f_2)(t,x)| \left[ (\varphi\sigma)(t,x+y) \sigma(t,x) - \varphi(x) \sigma(t,x)^2 - (\varphi\sigma)(t,y) \sigma(t,x) \right] \\
& \qquad \le |(f_1-f_2)(t,x)| \left[ \varphi(x+y) - \varphi(x) + \varphi(y) \right] \\
& \qquad \le 2 \varphi(y) |(f_1-f_2)(t,x)|\ ,
\end{align*}
we further obtain
\begin{align*}
& \frac{d}{dt} \int_0^\infty |(f_1-f_2)(t,x)| \varphi(x)\ dx \\
& \qquad \le \int_0^\infty \int_0^\infty K(x,y) (f_1+f_2)(t,y)  |(f_1-f_2)(t,x)| \varphi(y) \ dydx \\
& \qquad \le \int_0^\infty \int_0^\infty \varphi(x) \varphi(y)^2 (f_1+f_2)(t,y) |(f_1-f_2)(t,x)|\ dydx \\
& \qquad = \int_0^\infty \varphi(y)^2 (f_1+f_2)(t,y)\ dy\  \int_0^\infty \varphi(x) |(f_1-f_2)(t,x)|\ dx\ ,
\end{align*}
and the conclusion follows by Gronwall's inequality. 
\end{proof}

The second approach stems from the study of the well-posedness of \eqref{sc1}-\eqref{sc2} for the constant kernel $K_0(x,y)=2$, the additive kernel $K_1(x,y)=x+y$, and the multiplicative kernel $K_2(x,y)=xy$ performed in \cite{MP04}. Roughly speaking, there is a uniqueness result for the kernel $K_i$ in the weighted space $L^1(0,\infty;x^i dx)$, $i=0,1,2,$. It is worth pointing out that the homogeneity of the weight matches that of the coagulation kernel. Though the main tool used in \cite{MP04} is the Laplace transform, an alternative argument involving a weighted Wasserstein distance has been developed in \cite{FL06} to extend this uniqueness result to a wider class of homogeneous kernels with arbitrary homogeneity. For simplicity, we restrict ourselves to the coagulation kernel 
\begin{equation}
K(x,y) = x^\alpha y^\beta + x^\beta y^\alpha\ , \quad x>0\ , \ y>0\ , \label{sc62}
\end{equation} 
and assume that its homogeneity $\lambda:=\alpha+\beta$ lies in $(0,1]$. We refer to \cite{FL06} for more general assumptions on $K$ and $f^{in}$ and homogeneities in $(-\infty,0)$ or in $(1,2)$. 

\begin{proposition}\label{pr.sc5}
Let $T>0$ and $f^{in}\in L_\lambda^1(0,\infty)$, $f^{in}\ge 0$. There is at most one solution $f\in C([0,T);w-L_\lambda^1(0,\infty))$ to \eqref{sc1}-\eqref{sc2} (recall that the space $L_\lambda^1(0,\infty)$ is defined in \eqref{sc3}).
\end{proposition}

\begin{proof} Let $f_1$ and $f_2$ be two solutions to \eqref{sc1}-\eqref{sc2} enjoying the properties listed in Proposition~\ref{pr.sc5}. For $i=1,2$, we introduce the cumulative distribution function $F_i$ of $f_i$ given by
$$
F_i(t,x) := \int_x^\infty f_i(t,y)\ dy\ , \quad t>0\ , \ x>0\ ,
$$
and set $E:=F_1-F_2$ and
$$
R(t,x) := \int_0^x z^{\lambda-1} \text{ sign}(F_1-F_2)(t,z)\ dz\ , \quad t>0\ , \ x>0\ .
$$
We infer from \eqref{sc1} after some computations (see \cite[Proposition~3.3]{FL06}) that 
\begin{equation}
\frac{d}{dt} \int_0^\infty x^{\lambda-1} |E(t,x)|\ dx \le \frac{1}{2} \left( A_1(t)+A_2(t) \right)\ , \label{sc63}
\end{equation}
where
$$
A_1(t) := \int_0^\infty \int_0^\infty K(x,y) \left[ (x+y)^{\lambda-1} - x^{\lambda-1} \right]  (f_1+f_2)(t,y) |E(t,x)|\ dydx
$$
and 
\begin{align*}
A_2(t) & := \int_0^\infty \int_0^\infty \partial_x K(x,y) \tilde{R}(t,x,y) (f_1+f_2)(t,y) E(t,x)\ dydx\ , \\
\tilde{R}(t,x,y) & := R(t,x+y) - R(t,x) - R(t,y)\ .
\end{align*}

On the one hand, since $\lambda<1$,
$$
K(x,y) \left[ (x+y)^{\lambda-1} - x^{\lambda-1} \right] \le 0\ , \quad x>0\ , \ y>0\ ,
$$
so that
\begin{equation}
A_1(t) \le 0 \ . \label{sc64}
\end{equation}
On the other hand, it follows from the definition of $R$ and the subadditivity of $x\mapsto x^\lambda$ that
\begin{align*}
\left| \tilde{R}(t,x,y) \right| & = \left| \int_{\max\{x,y\}}^{x+y} z^{\lambda-1} \text{ sign}(E)(t,z)\ dz + \int_0^{\min\{x,y\}} z^{\lambda-1} \text{ sign}(E)(t,z)\ dz \right| \\
& \le \frac{1}{\lambda} \left[ (x+y)^\lambda - \max\{ x,y \}^\lambda + \min\{ x,y \}^\lambda \right] \\
& \le \frac{2}{\lambda} \min\{ x,y \}^\lambda\ .
\end{align*}
We deduce from the previous inequality and \eqref{sc62} that there is $C_5>0$ depending only on $\alpha$ and $\beta$ such that
$$
|\partial_x K(x,y)| \left| \tilde{R}(t,x,y) \right| \le C_5\ x^{\lambda-1} y^\lambda\ .
$$
Consequently, 
\begin{equation}
A_2(t) \le C_5 \ M_\lambda((f_1+f_2)(t))\ \int_0^\infty x^{\lambda-1} |E(t,x)|\ dx\ . \label{sc65}
\end{equation}
Collecting \eqref{sc63}-\eqref{sc65} we end up with
$$
\frac{d}{dt} \int_0^\infty x^{\lambda-1} |E(t,x)|\ dx \le \frac{C_5}{2}\ M_\lambda((f_1+f_2)(t))\ \int_0^\infty x^{\lambda-1} |E(t,x)|\ dx\ ,
$$
and the conclusion follows by integration. 
\end{proof}

Even though there is a version of Proposition~\ref{pr.sc5} when $K$ is given by \eqref{sc62} with $\lambda\in (1,2]$ (and is thus a gelling kernel), the requirement $f\in C([0,T);w-L_\lambda^1(0,\infty))$ is only true for $T<T_{gel}$ and thus provides no clue about uniqueness past the gelation time.

The only result we are aware of which deals with the uniqueness of solutions exhibiting a gelation transition is available for the multiplicative kernel $K_2(x,y)=xy$. In that particular case, using the Laplace transform, it is possible to characterize $M_1(f(t))$ for all times, prior and past the gelation time, and this information allows one to prove uniqueness \cite{Du94b, Ko88, NZ11, vRS06, SvR01}.

\section*{Acknowledgments}

These notes grew out from lectures I gave at the Institut f\"ur Angewandte Mathematik, Leibniz Universit\"at Hannover, in September 2008 and at the African Institute for Mathematical Sciences, Muizenberg, in July 2013. I thank Jacek Banasiak, Joachim Escher, Mustapha Mokhtar-Kharroubi, and Christoph Walker for their kind invitations as well as both institutions for their hospitality and support. This work was completed while visiting the Institut Mittag-Leffler, Stockholm.



\end{document}